\documentclass[11pt]{article}
\usepackage{amsfonts}
\usepackage{amssymb}
\usepackage{amsthm}
\usepackage{amsmath}
\usepackage{graphicx}
\usepackage{empheq}
\usepackage{indentfirst}
\usepackage{cite}
\usepackage{mathrsfs}
\usepackage{graphics}
\usepackage{enumerate}
\usepackage{color}

 \newtheorem{theorem}{Theorem}[section]
 \newtheorem{corollary}{Corollary}[section]
 \newtheorem{lemma}{Lemma}[section]
 \newtheorem{proposition}{Proposition}[section]

 \newtheorem{remark}{Remark}[section]
 \numberwithin{equation}{section}

\allowdisplaybreaks[4]

\def\e{\varepsilon}

\newcommand{\beq}{\begin{equation}}
\newcommand{\eeq}{\end{equation}}
 \def\non{\nonumber }
\def\bea{\begin{eqnarray}}
\def\eea{\end{eqnarray}}

\def\w{\tilde{w}}

\begin{document}
\title{Boundedness of Classical Solutions to a Degenerate Keller--Segel Type Model with Signal-dependent Motilities}

\author{Kentaro Fujie\thanks{Research Alliance Center for Mathematical Sciences, Tohoku University, Sendai 980-8578, Miyagi, Japan, \textsl{fujie@tohoku.ac.jp}},
	\ Jie Jiang\thanks{Innovation Academy for Precision Measurement Science and Technology, CAS,
		Wuhan 430071, HuBei Province, P.R. China,
		\textsl{jiang@wipm.ac.cn, jiang@apm.ac.cn}.}}
\date{\today}

\maketitle

\begin{abstract} 
In this paper, we consider the initial Neumann boundary value problem for a degenerate kinetic model of Keller--Segel type. The system features a signal-dependent decreasing motility function that vanishes asymptotically, i.e., degeneracies may take place as the concentration of signals tends to infinity. 
In the present work, we are interested in the boundedness of classical solutions when the motility function satisfies certain decay rate assumptions. 
Roughly speaking, in the two-dimensional setting, we prove that classical solution is globally bounded  if the motility function decreases slower than an exponential speed at high signal concentrations. In higher dimensions, boundedness is obtained when the motility decreases at certain algebraical speed. The proof is based on the comparison methods developed in our previous work \cite{FJ19a,FJ19b} together with a  modified Alikakos--Moser type iteration. Besides, new  estimations involving certain weighted energies are also constructed to establish the boundedness.

\end{abstract}	

\section{Introduction}	
This paper is a continuation of our previous work \cite{FJ19a,FJ19b} on the the following initial boundary value problem:
	\begin{equation}
	\begin{cases}\label{chemo1}
	u_t=\Delta (\gamma (v)u)&x\in\Omega,\;t>0,\\
	\e v_t-\Delta v+v=u&x\in\Omega,\;t>0,\\
	\partial_\nu u=\partial_\nu v=0,\qquad &x\in\partial\Omega,\;t>0,\\
	u(x,0)=u_0(x),\;\;\e v(x,0)=\e v_0(x),\qquad & x\in\Omega.
	\end{cases}
	\end{equation}
  This model was recently proposed in \cite{PRL12,Sciencs11} to describe the process of pattern formations via the so-called ``self-trapping" mechanism. Here, $u$ and $v$ stand for the density of cells and the concentration of signals, respectively. The cellular motility $\gamma(\cdot)$ was assumed to be suppressed by the concentration of signals, which characterizes the incessant tumbling of cells at high concentration.  
  In other words, $\gamma(\cdot)$ is a  signal-dependent decreasing function, i.e., $\gamma'(v)\leq 0$. Hence, the system features a vanishing macroscopic motility as $v$ becomes unbounded.
  
   This model can also be regarded as a special version of Keller--Segel type model with signal-dependent diffusion rates and chemo-sensitivities introduced  by Keller and Segel in their seminal works \cite{KSa,KSb,KSc}. In fact, in \cite{KSb}, the evolution of cell density was described by the following  equation:
   \begin{equation}\label{ksv}
   	u_t=\nabla \cdot(\mu(v)\nabla u-u\chi(v)\nabla v),
   \end{equation}where  the cell diffusion rate $\mu$ and chemo-sensitivity $\chi$ are linked via
   \begin{equation}\label{ksv0}
   	\chi(v)=(\sigma-1)\mu'(v).
   \end{equation}
On the other hand, a direct decomposition of the right-hand side of the first equation in \eqref{chemo1} yields  the following variant form
 \begin{equation}
 \begin{cases}
 u_t=\nabla \cdot (\gamma (v)\nabla u)+\nabla\cdot(u\gamma'(v)\nabla v),\\
 \e v_t-\Delta v+v=u,
 \end{cases}
 \end{equation}
 which corresponds to the special case of \eqref{ksv} with $\sigma=0$ in \eqref{ksv0}. Recall that  the parameter $\sigma$ is proportional to the distance between chemical receptors in the cells.
 In the case $\sigma=0$, the distance between receptors is zero. In other words, chemotaxis occurs because of an undirected effect on activity due to the presence of a chemical sensed by a single receptor (local sensing), which is distinct from the directed chemotactic movement when $\sigma>0$ attribute to gradient sensing mechanism by comparing the chemical concentrations at different spots.

 Theoretical analysis for problem \eqref{chemo1} has been carried out recently in several work, see e.g., \cite{Anh19,JKW18,LW2020,LX2019,WW2019,YK17,TaoWin17,BLT20}, where the major degeneracy issue was tackled basically by energy method. In fact, standard elliptic/parabolic regularity theory tells that $L^\infty_tL^p_x$-boundedness of $u$ with any $p>\frac{n}{2}$ will yield to an upper bound of $v$. Thus, an indirect way to prevent degeneracy is to establish higher integrability of $u$. However, this idea seems only efficient for some specific cases, where several additional assumptions are needed to achieve the $L^\infty_tL^p_x$-boundedness of $u$. For example, smallness of some coefficients \cite{YK17}, particular choices of the motility functions \cite{Anh19,YK17}, or a presence of logistic source term in the first equation \cite{JKW18,LW2020,LX2019,WW2019}, etc.

In contrast, a new comparison method based on a careful observation of the delicate nonlinear structure  was developed to establish the upper bound of $v$ directly in our previous work \cite{FJ19a,FJ19b}. We proved that in any spatial dimensions and with any motility function satisfying $\mathrm{(A0)}$ when $\e=0$, or additionally $\mathrm{(A1)}$ when $\e>0$ below,  the upper bound of $v$ grows at most exponentially in time and thus degeneracy cannot take place in finite time. Then  we showed that classical solution always exists globally in dimension two. Moreover, under certain polynomial growth condition on $1/\gamma$, we discussed  uniform-in-time boundedness when $n=2,3$.  More importantly, occurrence of exploding solutions was examined for the first time. In the case $\gamma(v)=e^{-v}$, a novel critical-mass phenomenon in the  two-dimensional setting was observed  that with any sub-critical mass, the global solution is uniformly-in-time bounded while with certain super-critical mass, the global solution  will blow up at time infinity. We mention that in the special case $\gamma(v)=e^{-v}$,  global boundedness with sub-critical mass and possible blowup at unspecified blow-up time with super-critical mass was also proved in \cite{JW20} by energy method. In \cite{BLT20}, the authors also verified that blowup of classical solution must occur at time infinity by duality method and moreover, weak solutions was obtained in any dimensions when $\gamma(v)=e^{-v}$.

In the present work, we aim to continue our discussion on uniform-in-time boundedness of classical solutions with generic motility functions and arbitrarily large initial data. Throughout this paper, we assume  $\Omega\subset\mathbb{R}^n$ with $n\geq2$  being a smooth bounded domain and
\begin{equation}\label{ini}
(u_0,v_0)\in C^0(\overline\Omega)\times{W^{1,\infty}(\Omega)},\quad u_0\geq0,\; v_0>0 \quad  \mbox{in } \overline\Omega, \quad u_0\not\equiv0.
\end{equation}
For  $\gamma$, we first require in general that
\begin{equation}\label{gamma0a}
\mathrm{(A0)}:\gamma(v)\in C^3[0,+\infty),\;\gamma(v)>0,\;\;\gamma'(v)\leq0\;\;\text{on}\;(0,+\infty).
\end{equation}
Additionally if $\e>0$, we need  the following asymptotic property:
\begin{equation}\label{gamma}
\mathrm{(A1)}:\lim\limits_{s\rightarrow+\infty}\gamma(s)=0.
\end{equation}
Moreover,  in order to study the uniform-in-time boundedness, we will propose certain decreasing assumptions on $\gamma$. In particular  we shall consider the form $\gamma(s)=s^{-k}$ with some $k>0$ as a toy model. In this case, the variant form reads
\begin{equation}\label{variant2}
 \begin{cases}
u_t=\nabla\cdot\left[v^{-k}(\nabla u-ku\nabla \log v)\right],\\
\e v_t-\Delta v+v=u,
\end{cases} 
\end{equation}
which resembles the classical Keller--Segel model with a logarithmic chemo-sensitivity:
\begin{equation}\label{logKS}
\begin{cases}
u_t=\nabla\cdot (\nabla u-ku\nabla \log v),\\
\varepsilon v_t=\Delta v-v+u.
\end{cases}
\end{equation}

Up to now, theoretical results on \eqref{logKS} are far from satisfactory. Roughly speaking, existence of global solutions or blowups seems to be determined by the size of $k$. Blowup solution was constructed only  in the radial symmetric case when $\e=0$,  $n\geq3$ and $k>\frac{2n}{n-2}$ \cite{Nagai98}. On the other hand, there are several attempts on enlarging the admissible range of $k$ for global existence and however, the threshold number is still unclear. Since a complete description of related results can be found in \cite{FS2016,fs2018, LanWin}, we omit a detailed review here.

For the degenerate model under consideration, the  result in our work \cite{FJ19a,FJ19b} in the two-dimensional setting asserts that any polynomial decreasing motility would give rise to globally boundedness whereas the exponentially decaying type will result in a critical mass phenomenon,  which strongly indicates that the dynamic of solutions is closely related to the decay rate of $\gamma(\cdot)$. In this regard, one motivation of the current work is to understand the connection between decay rate and global existence or boundedness. In the context of  particular choice $\gamma(s)=s^{-k}$, it suffices to find out an admissible range of $k$ as well. Note that systems \eqref{variant2} and \eqref{logKS} share the same set of equilibria. Thus a study on our system with $\gamma(s)=s^{-k}$ may also lead us  to a better understanding of the mechanism of the logarithmic Keller--Segel model.

Now we recall some related results on the toy model case \eqref{variant2}. If  $\e=0$, global classical solution with uniform-in-time bounds was obtained by delicate energy estimates in \cite{Anh19} when $n\leq 2$ for any $k>0$ or $n\geq3$ for $k<\frac{2}{n-2}$ (see \cite{Wang20} for an alternative proof; see also \cite{YK17} for global existence under certain smallness assumptions;). In our previous work \cite{FJ19a,FJ19b}, we generalized above boundedness results in dimension two for any $\e\geq0$ under the following weaker assumption:
 \begin{equation}\label{gamma2}\mathrm{(A2)}:\qquad\text{there is $k>0$ such that}
 \lim\limits_{s\rightarrow+\infty}s^{k}\gamma(s)=+\infty.
 \end{equation}
More precisely, we proved that if $n=2$, then problem \eqref{chemo1} with any $\e\geq0$ has a uniformly-in-time bounded classical solution provided that $\gamma$ satisfies $\mathrm{(A0)}-\mathrm{(A1)}$ as well as $\mathrm{(A2)}$. Note that $\mathrm{(A2)}$ allows $\gamma$ to take any decreasing form within a finite region and moreover, other algebraically decreasing functions are permitted as well, for example, $\gamma(v)=\frac{1}{v^{k}\log(1+v)}$ with any $k>0$. Furthermore, if $n=3$ and $\e>0$, we also obtained globally bounded solution provided additionally that
\begin{equation}
\label{A3}\mathrm{(A3)}:\qquad 2|\gamma'(s)|^2\leq \gamma(s)\gamma''(s),\;\;\forall\;s>0.
\end{equation}
Under assumptions $\mathrm{(A0)}$, $\mathrm{(A1)}$ and $\mathrm{(A3)}$, $1/\gamma(s)$ can grow at most linearly in $s$. Correspondingly, if we take $\gamma(s)=s^{-k}$, then $\mathrm{(A3)}$ will yield to a constraint $k\leq1$.

 Now we are in a position to state the main results of the current work. First, we give uniform-in-time boundedness for the two dimensional case. Note  in previous work \cite{FJ19a,FJ19b,Anh19}, $\gamma$ can decrease at most algebraically in $v$. 
 \begin{theorem}\label{TH0}
 	Assume $n=2$. Suppose that $\gamma$ satisfies $\mathrm{(A0)}$ if $\e=0$ and additionally $\mathrm{(A1)}$ if $\e>0$. Moreover, suppose that
 \begin{equation}
 \label{A2'}\mathrm{(A2')}:\qquad
 \lim\limits_{s\rightarrow+\infty}e^{\alpha s}\gamma(s)=+\infty,\;\;\forall\;\alpha>0.
 \end{equation}
 Then problem \eqref{chemo1} has a unique global classical solution which is uniformly-in-time bounded.
 
 On the other hand, if there is $\chi>0$ such that\begin{equation}
 \label{A2''}\mathrm{(A2'')}:\qquad
 \lim\limits_{s\rightarrow+\infty}e^{\chi s}\gamma(s)=+\infty,
 \end{equation}then the solution of \eqref{chemo1} is uniformly-in-time bounded provided that $\|u_0\|_{L^1(\Omega)}<\frac{4\pi}{\chi}$.
 \end{theorem}

 \begin{remark}\label{remTH0}Recall that in \cite{FJ19a,FJ19b}, we have proved that in the two dimensional setting, classical solution always exists globally provided that $\gamma$ satisfies $\mathrm{(A0)}$ if $\e=0$ and additionally $\mathrm{(A1)}$ if $\e>0$. In addition, when $\gamma=e^{-\chi v}$ for any $\chi>0$, there is a critical-mass phenomenon  with the threshold number being $\frac{4\pi}{\chi}$ for general bounded domains. More precisely, the global classical solution is uniform-in-time bounded with any sub-critical mass while with certain super-critical mass, the solution will become unbounded as time goes to infinity.
 	
Note that	$\mathrm{(A2')}$ is weaker than $\mathrm{(A2)}$. 
 	For example, $\gamma (s) = e^{-\chi s^{\beta}}$ with any $\chi>0$ and $0<\beta<1$ is excluded from $\mathrm{(A2)}$, but satisfies $\mathrm{(A2')}$. In this regard, Theorem \ref{TH0}  partially indicates that in 2D the exponentially decay rate of $\gamma$ is critical for global boundedness of the classical solutions  with large mass.
 \end{remark}

 Next, we consider higher dimensional cases.  For the parabolic-elliptic case $\e=0$, we obtain that
\begin{theorem}
	\label{TH1} Assume $n\geq3$ and $\e=0$. Suppose that $\gamma$ satisfies  $\mathrm{(A0)}$ and the following condition:	
	\begin{equation}
	\label{A3a}\mathrm{(A3a)}:\qquad\sqrt{\frac{n}{2}}|\gamma'(s)|^2< \gamma(s)\gamma''(s),\;\;\forall\;s>0.
	\end{equation}
Then problem \eqref{chemo1} has a unique global classical solution.

In addition, if $\gamma$ satisfies  $\mathrm{(A1)}$ and
	\begin{equation}
\label{A3u}\mathrm{(A3u)}:\qquad l_0|\gamma'(s)|^2\leq \gamma(s)\gamma''(s),\;\;\text{with some}
\;l_0>\frac{n}{2}\;\;\text{for all}\;s>0,
\end{equation}
 then the global solution is uniformly-in-time bounded.
\end{theorem}
\begin{remark}
	According to Lemma  \ref{lemA23} below, if $\gamma(\cdot)$ satisfies  $\mathrm{(A0)}$, $\mathrm{(A1)}$ and  $\mathrm{(A3u)}$, then it must fulfill assumption $\mathrm{(A2)}$ with  some $k<\frac{2}{n-2}$. In this regard, our uniform-in-time boundedness result covers those in \cite{Anh19} established for the special case $\gamma(s)=s^{-k}$ with any $0<k<\frac{2}{n-2}$.
\end{remark}
On the other hand when $\e>0$, we prove the following boundedness result.
\begin{theorem}\label{TH2}Assume $n\geq3$ and $\e>0$. Suppose that $\gamma$ satisfies  $\mathrm{(A0)}-\mathrm{(A1)}$  and the following condition:	
	\begin{equation}
	\label{A3b}\mathrm{(A3b)}:\qquad\bigg(1+[\frac{n}{2}]\bigg)|\gamma'(s)|^2\leq \gamma(s)\gamma''(s),\;\;\forall\;s>0,
	\end{equation}
	where $[\frac{n}{2}]$ denotes the maximal integer less or equal to $\frac{n}{2}$.
	Then problem \eqref{chemo1} has a unique global classical solution, which is uniformly-in-time bounded.
\end{theorem}
\begin{remark}
	For the sake of simplicity, we normalize all physical parameters except $\e$ and in the proof we take $\e=1$ for the case $\e>0$. But the statements of our results and assumptions $\mathrm{(A3a)}$, $\mathrm{(A3u)}$ and $\mathrm{(A3b)}$ are independent of the choice of parameters.
\end{remark}
\begin{remark}
Since $v_0>0$ in $\overline{\Omega}$,	thanks to the strictly positive time-independent lower bound $v_*$ of $v$ for $(x,t)\in\overline{\Omega}\times[0,\infty)$ given in  Lemma \ref{lbdw} and Lemma \ref{lbdv} in the next section, our existence and boundedness results also hold true if $\gamma(s)$ has singularities at $s=0$, for example $\gamma(s)=s^{-k}$ with $k>0$. In such cases, we can simply replace $\gamma(s)$ by a new motility function $\tilde{\gamma}(s)$ which satisfies $\mathrm{(A0)}$ and coincides with $\gamma(s)$ for $s\geq\frac{v_*}{2}$.
\end{remark}
In particular, for the typical case $\gamma(v)=v^{-k}$, we have
\begin{theorem}
	Suppose  that $\gamma(v)=v^{-k}$ and $n\geq3$. Then,
	\begin{itemize}
		\item 	when $\e=0$, problem \eqref{chemo1} has a unique global classical solution provided that $k<\frac{\sqrt{2n}+2}{n-2}$. In addition, the global solution is uniformly-in-time bounded if $k<\frac{2}{n-2}$;
		\item when $\e>0$, problem $\eqref{chemo1}$ has a uniformly-in-time bounded global solution provided that $k\leq 1/[\frac{n}{2}]$.
	\end{itemize}
\end{theorem}

Now, let us sketch the main idea of our proof for boundedness in higher dimensions. First, it is necessary to briefly recall some related results in our  work \cite{FJ19a,FJ19b}. Denote  $w(x,t)$  the unique non-negative solution of the following Helmholtz equation:
\begin{equation*}
\begin{cases}
-\Delta w+w=u, &x\in\Omega,\;t>0,\\
\partial_\nu w=0,&x\in\partial\Omega,\;t>0.
\end{cases}
\end{equation*} 
Then we found that (\cite[Lemma 3.1]{FJ19a} or \cite[Lemma 4.1]{FJ19b})
\begin{equation}\label{keyid0}
w_t+\gamma(v)u=(I-\Delta)^{-1}[\gamma(v)u],
\end{equation}
which unveils the intrinsic mechanism of the nonlinear structure.  Here, $(I-\Delta)^{-1}$ denotes the inverse operator of $I-\Delta$ and $\Delta$ is the Laplacian operator with homogeneous Neumann boundary condition.
Using comparison principle of the elliptic equations together with Gronwall's inequality, we proved from the above key identity that (\cite[Lemma 3.2]{FJ19a} or \cite[Lemma 4.1]{FJ19b})
\begin{equation*}
w(x,t)\leq w_0(x)e^{Ct},\;\;\text{for all}\;x\in\Omega\;\;\text{and}\;t\geq0,
\end{equation*}
where  $w_0\triangleq(I-\Delta)^{-1}u_0$ and $C>0$ depends only on $\gamma, \Omega$ and the initial data.  Note that in the parabolic-elliptic case, i.e., $\e=0$ in \eqref{chemo1}, $w$ is identical to $v$. On the other hand when $\e>0$, thanks to the above identity again, upon an application of the comparison principle for parabolic equations, we proved that (\cite[Lemma 4.3]{FJ19b})
\begin{equation}
	v(x,t)\leq C(w(x,t)+1)\;\;\text{for all}\;x\in\Omega\;\;\text{and}\;t\geq0,
\end{equation}
with  $C>0$ depending only on $\gamma, \Omega$ and the initial data. In a word, $v(x,t)$ can  grow point-wisely at most exponentially in time in both cases. 

In addition, under certain decay assumptions for example, $\mathrm{(A2')}$ when $n=2$, or  $\mathrm{(A2)}$ with some $k<\frac{2}{n-2}$ when $n\geq3$, the above upper bound estimate can be further improved. Take $\e=0$ and $n\geq3$ for example and recall that $w=v$ in this case. In \cite{FJ19a,FJ19b}, time-independent upper bounds of $v$ were proved directly when $n\leq3$ by simple arguments based on an application of the uniform Gronwall inequalities. However, since we made use of the Sobolev embedding $H^2\hookrightarrow L^\infty$ there, the technique fails in higher dimensions. In this work, observing that the key identity also  reads ($\e=0$)
\begin{equation}\label{kid0}
v_t-\gamma(v)\Delta v+v\gamma(v)=(I-\Delta)^{-1}[\gamma(v)u],
\end{equation}
we develop an alternative approach based on a delicate Alikakos--Moser type iteration to achieve the same goal in higher dimensions. More precisely, for any $n\geq3$ we are able to prove that under the assumptions $\mathrm{(A0)}$ and $\mathrm{(A2)}$ with any $k<\frac{2}{n-2}$ when $\e>0$, or additional $\mathrm{(A1)}$ when $\e>0$, $v$ has a time-independent upper bound; see Proposition \ref{Propunbv}. 
Note here the uniform-in-time upper bound of $v$ is obtained {\it independently} of $u$ under a much weaker decay rate assumption than that in \cite{Anh19}.

In order to establish the global existence or time-independent boundedness, it remains to derive $L^\infty_tL^p_x$ (time-independent) boundedness of $u$ with some $p>\frac{n}{2}$ due to standard bootstrap argument. Here, the key idea is to construct an estimate for a weighted energy $\int_\Omega u^{p}\gamma^{q}(v)$ with some $p>\frac{n}{2}$ and $q>0$. Since $v$ is bounded from above now, $\gamma(v)$ is bounded from below thanks to its decreasing property. Then $L^\infty_tL^p_x$ boundedness of $u$ follows from the boundedness of  the above weighted  energy.  Adjusting the parameters $p,q$ carefully and using the key identity again, we are able to construct a  new estimation involving the weighted energy which gives rise to the desired boundedness.

We remark that at the present stage, we cannot obtain boundedness results for the case $\e>0$ under the same condition $\mathrm{(A3u)}$ as for the case $\e=0$. The main obstacle comes from  the different equations for $v_t$, where an additional diffusion coefficient $\gamma(v)$  in \eqref{kid0}  helps to weaken the constraint when $\e=0.$ Besides, in the fully parabolic case $\e>0$, we cannot simply adjust $p,q$ in a single estimation involving $\int_\Omega u^{p}\gamma^{q}(v)$ to get the desired result as done for the case $\e=0.$ A different strategy used  here is to list out a system of estimations involving the weighted energies with $p=2,3,...,1+[\frac{n}{2}]$ and $q=0,1,...,p-1$. Then by a careful recombination of such estimations and an iteration argument together with an application of the uniform Gronwall inequality, we prove the time-independent boundedness of the weighted energies.

Before concluding this part, we would like to stress some new features of the present work. Firstly, we improves the boundedness result with arbitrarily large initial data in dimension two, which partially indicates that the exponential decay case is critical for boundedness with large mass. We remark that it is still unknown  whether the 2-D global classical solution would be bounded or blow up at time infinity with large initial data if $\gamma$ decays at a speed faster than exponential rate. Secondly,  uniform boundedness for $v$ is independently proved provided that $\gamma$ satisfies $\mathrm{(A2)}$ with some $k<\frac{2}{n-2}$  when $n\geq3$ by delicate iterations.  For the case $\e=0$, boundedness of $u$ is achieved under a slightly stronger assumption $\mathrm{(A3u)}$. Here, our work also provides an alternative proof for the result in \cite{Anh19} concerning the particularly chosen motility $\gamma(v)=v^{-k}$ with any $k<\frac{2}{n-2}$. Lastly, to the best of our knowledge, boundedness given in Theorem \ref{TH2}  is the first result for the case for $\e>0$ and $n\geq3.$ It is still challenging whether one can prove boundedness of $u$ under the same decay condition as for $v$, or the slightly stronger one $\mathrm{(A3u)}$.

The rest of the paper is organized as follows.  In Section 2, we provide some preliminary results and recall some useful lemmas. Then in Section 3 we use modified Alikakos--Moser iteration to derive the uniform-in-time upper bounds of $v$. In Section 4, we study the parabolic-elliptic case $\e=0$ and establish the boundedness of weighted energy. In Section 5, we prove boundedness of weighted energy for the fully parabolic case.

\section{Preliminaries}	
	In this section, we recall some useful lemmas. First, local existence and uniqueness of classical solutions to system \eqref{chemo1} can be
	established by the standard fixed point argument and  regularity theory for elliptic/parabolic equations. Similar proof can be found in \cite[Lemma 3.1]{Anh19} or \cite[Lemma 2.1]{JKW18}  and hence here we omit the detail here.
	\begin{theorem}\label{local}
		Let $\Omega$ be a smooth bounded domain of $\mathbb{R}^n$. Suppose that $\gamma(\cdot)$ satisfies $\mathrm{(A0)}$ and $(u_0,v_0)$ satisfies \eqref{ini}. Then there exists $T_{\mathrm{max}} \in (0, \infty]$ such that problem \eqref{chemo1} permits a unique non-negative classical solution $(u,v)\in (C^0(\overline{\Omega}\times[0,T_{\mathrm{max}}))\cap C^{2,1}(\overline{\Omega}\times(0,T_{\mathrm{max}})))^2$. Moreover, the following mass conservation holds
		\begin{equation*}
		\int_{\Omega}u(\cdot,t)dx=\int_{\Omega}u_0 dx
		\quad \text{for\ all}\ t \in (0,T_{\mathrm{max}}).
		\end{equation*}	
		If $T_{\mathrm{max}}<\infty$, then

		\begin{equation*}
	\limsup\limits_{t\nearrow T_{\mathrm{max}}} 
		\|u(\cdot,t)\|_{L^\infty(\Omega)}=\infty.
		\end{equation*}
		\end{theorem}
In the same manner as to classical Keller--Segel systems, we can prove the following criterion (see e.g., \cite[Lemma 4.3]{Anh19}).
\begin{lemma}\label{criterion}
	For any $p>\frac{n}{2}$, if the solution of \eqref{chemo1} satisfies that
	\begin{equation}
		\|u(\cdot,t)\|_{L^p(\Omega)}\leq C,\;\;\text{for all} \;t\in(0,T_{\mathrm{max}})
	\end{equation}with some $C>0$,
	then  $T_{\mathrm{max}}=\infty $ and there holds
	\begin{equation}
		\sup\limits_{t>0}\bigg(\|u(\cdot,t)\|_{L^\infty(\Omega)}+\|v(\cdot,t)\|_{W^{1,\infty}(\Omega)}\bigg)\leq C'
	\end{equation}with some $C'>0$. Moreover, if the above constant $C>0$ is time-independent, then the global solution has a uniform-in-time bound as well.	
\end{lemma}

	Next, as done in our previous work \cite{FJ19a,FJ19b}, we introduce an auxiliary variable $w(x,t)$, which is the unique non-negative solution of the following Helmholtz equation:
	\begin{equation*}
	\begin{cases}
	-\Delta w+w=u, &x\in\Omega,\;t>0,\\
	\partial_\nu w=0,&x\in\partial\Omega,\;t>0.
	\end{cases}
	\end{equation*} 
Now, we recall the following lemma given in \cite{Anh19} about estimates for the solution of Helmholtz equations. Let $a_+=\max\{a,0\}$. Then we have
	\begin{lemma}\label{lm2}
		Let $\Omega$ be  a smooth bounded domain in $\mathbb{R}^n$, $n\geq1$ and let $f\in C(\overline{\Omega})$ be a non-negative  function such that $\int_\Omega f dx>0$. If $z$ is a $C^2(\overline{\Omega})$ solution to
		\begin{equation}
		\begin{split}\label{helm}
		-\Delta z+z=f,\;\;x\in\Omega,\\
		\frac{\partial z}{\partial \nu}=0,\;\;x\in\partial\Omega,
		\end{split}
		\end{equation}then if $1\leq q< \frac{n}{(n-2)_+}$, there exists a positive constant $C=C(n,q,\Omega)$ such that
		\begin{equation}
		\|z\|_{L^q(\Omega)}\leq C\|f\|_{L^1(\Omega)}.
		\end{equation}
	\end{lemma}
When $n=2$, we need the following result given in \cite[Lemma 3.3]{Wang20}, see also \cite[Lemma A.3]{TaoWin14}.
\begin{lemma}\label{lm2e}
	Let $\Omega\subset\mathbb{R}^2$ be a smooth bounded domain. For any  $f\in L^1(\Omega)$ such that\begin{equation}
		\|f\|_{L^1(\Omega)}= \Lambda,
	\end{equation}with some $\Lambda>0$, there is $C>0$ such that the solution of \eqref{helm} satisfies
	\begin{equation}
		\int_\Omega e^{Az}dx\leq C
	\end{equation}for any $0<A<\frac{4\pi}{\Lambda}.$
\end{lemma}

Besides, a strictly positive uniform-in-time lower bound for $w=(I-\Delta)^{-1}[u](x,t)$ is given by the positivity of the Green function to the Helmholtz equation (\cite{ito}) and the mass conservation. See also \cite[Lemma 3.3]{Black}.
	\begin{lemma}\label{lbdw}
		Suppose $(u,v)$  is the classical solution of \eqref{chemo1} up to the maximal time of existence $T_{\mathrm{max}}\in(0,\infty]$. Then, there exists a strictly positive constant $w_*=w_*(n,\Omega,\|u_0\|_{L^1(\Omega)})$ such that for all $t\in(0,T_{\mathrm{max}})$, there holds
		\begin{equation*}
		\inf\limits_{x\in\Omega}w(x,t)\geq w_*.
		\end{equation*}
	\end{lemma}
Similarly, a strictly positive uniform-in-time lower bound for $v$ was given in \cite[Lemma 2.1]{FS2016} provided that $v_0$ is strictly positive in $\overline{\Omega}$.
	\begin{lemma}\label{lbdv}
		Assume that $(u_0,v_0)$ satisfies \eqref{ini}. If $(u,v)$ is the solution of \eqref{chemo1} in $\Omega \times (0,T)$,
		then  there exists some $v_* >0$ such that
		\begin{eqnarray*}
			\inf_{x\in \Omega} v(x,t ) \geq v_*>0\qquad
			\mbox{for all }t\in(0,T).
		\end{eqnarray*}
		Here the constant $v_*$ is independent of $T>0$.
	\end{lemma}
By the comparison method developed in our previous work, we proved the following upper bounds for $w$ and $v$ (see \cite[Lemma 3.1]{FJ19a} and \cite[Lemma 4.1, Lemma 4.3 \& Remark 4.1]{FJ19b}).
\begin{lemma}\label{keylem1}Assume $n\geq1$ and suppose that $\gamma$ satisfies  $(\mathrm{A0})$.  For any $0<t<T_{\mathrm{max}}$, there holds
	\begin{equation}\label{keyid}
	w_t+\gamma(v)u=(I-\Delta)^{-1}[\gamma(v)u].
	\end{equation}
	Moreover, for  any $x\in\Omega$ and  $t\in[0,T_{\mathrm{max}})$, we have
	\begin{equation}\label{ptesta}
	w(x,t)\leq w_0(x)e^{\gamma(v_*)t}.
	\end{equation}
\end{lemma}	
	
	\begin{lemma}\label{vbd}
	 Assume that $\e>0$. Suppose $\gamma$ satisfies  $(\mathrm{A0})$ and the following asymptotic property:
	\begin{equation}\label{A1'}
	\mathrm{(A1')}:\lim\limits_{s\rightarrow+\infty}\gamma(s)<1/\e.
	\end{equation} 
	Then there exist $K>0$ depending on $\gamma$, $\e$ and the initial data and a generic constant $\tilde{C}>0$ independent of $\gamma$ such that for all $(x,t)\in\Omega\times[0,T_{\mathrm{max}})$,
		\begin{equation}\label{vbound}
		v(x,t)\leq \tilde{C}\bigg(w(x,t)+K\bigg).
		\end{equation}
	Furthermore, if $\gamma$ satisfies 	$\mathrm{(A1)}$ instead of $\mathrm{(A1')}$, then $\tilde{C}$ can be chosen as an arbitrary constant larger than $1$. 
	\end{lemma}
Finally, we need  the following uniform Gronwall inequality \cite[Chapter III, Lemma 1.1]{Temam} to deduce uniform-in-time estimates for the solutions.
\begin{lemma}\label{uniformGronwall}
	Let $g,h,y$ be three positive locally integrable functions on $(t_0,\infty)$ such that $y'$ is locally integrable on $(t_0,\infty)$ and the following inequalities are satisfied:
	\begin{equation*}
	y'(t)\leq g(t)y(t)+h(t)\;\;\forall\;t\geq t_0,
	\end{equation*}
	\begin{equation*}
	\int_t^{t+r}g(s)ds\leq a_1,\;\;\int_t^{t+r}h(s)ds\leq a_2,\;\;\int_t^{t+r}y(s)ds\leq a_3,\;\;\forall \;t\geq t_0,
	\end{equation*}	where $r,a_i$, $(i=1,2,3)$ are positive constants. Then
	\begin{equation*}
	y(t+r)\leq \left(\frac{a_3}{r}+a_2\right)e^{a_1},\;\;\forall t\geq t_0.
	\end{equation*}
\end{lemma}	
\section{Time-independent upper bounds of $v$}
This section is devoted to the following uniform-in-time boundedness result of $v$. In the two-dimensional case, the proof is based on a simple application of  the 2D Sobolev embeddings together with the uniform Gronwall inequality while in higher dimensions, the boundedness is achieved via a modified Alikakos--Moser type iteration argument.
\subsection{The two-dimensional case}
In two dimensions, it was proved in \cite{FJ19a,FJ19b} that global classical solution always exists provided that $\gamma$ satisfies  $\mathrm{(A0)}$ if $\e=0$ and additionally  $\mathrm{(A1)}$ if $\e>0$. In order to establish the boundedness, we first prove the following result.
\begin{lemma}\label{wH1a}
	Under the same assumptions of Theorem \ref{TH0}, there is $C>0$ depending only on  the initial data, $\gamma$, $\e$ and $\Omega$ such that
	\begin{equation}
	\sup\limits_{t\geq0}\left(\|\nabla w\|_{L^2(\Omega)}+\|w\|_{L^2(\Omega)}+\int_t^{t+1}\int_\Omega\gamma(v)u^2dxds\right)\leq C.
	\end{equation}
\end{lemma}
\begin{proof}
		Multiplying the key identity \eqref{keyid} by $u$ and recalling that $w=(I-\Delta)^{-1}[u]$, we obtain that
	\begin{equation*}
	\begin{split}
	\frac{1}{2}\frac{d}{dt}\left(\|\nabla w\|_{L^2(\Omega)}^2+\|w\|_{L^2(\Omega)}^2\right)+\int_\Omega \gamma(v)u^2dx=&\int_\Omega (I-\Delta)^{-1}[\gamma(v)u]udx\\
	=&\int_\Omega \gamma(v)uwdx\\
	\leq&\gamma(v_*)\int_\Omega uwdx.
	\end{split}
	\end{equation*}
{
	On the other hand, by integrating by parts it follows 
		\begin{equation*}
	\|\nabla w\|_{L^2(\Omega)}^2+\|w\|_{L^2(\Omega)}^2=\int_\Omega w udx.
	\end{equation*} 
Combining the above inequalities, we have
	\begin{equation*}
	\begin{split}
	&\frac{d}{dt}\left(\|\nabla w\|_{L^2(\Omega)}^2+\|w\|_{L^2(\Omega)}^2\right)
	+\|\nabla w\|_{L^2(\Omega)}^2+\|w\|_{L^2(\Omega)}^2
	+2\int_\Omega \gamma(v)u^2dx\\
	=&(2\gamma(v_*)+1)\int_\Omega w udx\\
	\leq&\int_\Omega \gamma(v)u^2dx
	+\frac{(2\gamma(v_*)+1)^2}{4}\int_\Omega\gamma^{-1}(v)w^2dx.
	\end{split}
	\end{equation*} 
Thus we obtain that
	\begin{equation}\label{wb00}
	\begin{split}
	&\frac{d}{dt}\left(\|\nabla w\|_{L^2(\Omega)}^2+\|w\|_{L^2(\Omega)}^2\right)
+\|\nabla w\|_{L^2(\Omega)}^2+\|w\|_{L^2(\Omega)}^2
	+\int_\Omega \gamma(v)u^2dx\\
	&\leq C\int_\Omega\gamma^{-1}(v)w^2dx
	\end{split}
	\end{equation} 
with some $C>0$.
}
	In view of our assumption { $\mathrm{(A2')}$}, we may infer that for any  $b>0$, $\alpha>0$, there exists $s_b>v_*$ depending on $\alpha$ and $b$ such that  for all $s\geq s_b$
	\begin{equation*}
	\gamma^{-1}(s)\leq be^{\alpha s}
	\end{equation*}and on the other hand, since $\gamma(\cdot)$ is decreasing,
	\begin{equation*}
	\gamma^{-1}(s)\leq \gamma^{-1}(s_b)
	\end{equation*}for all $0\leq s<s_b$.
	Therefore, for all $s\geq0$, there holds
	\begin{equation}\label{cond_gamma0}
	\gamma^{-1}(s)\leq be^{\alpha s}+\gamma^{-1}(s_b).
	\end{equation}
	Thus, we deduce from above and Lemma \ref{vbd}
	that	\begin{equation}\label{unic0}
	\begin{split}
	\int_\Omega\gamma^{-1}(v)w^2dx\leq &\int_\Omega (be^{\alpha v}+\gamma^{-1}(s_b))w^2dx\\
	\leq&\int_\Omega \left(be^{\tilde{C}\alpha(w+K)}+\gamma^{-1}(s_b)\right)w^2dx.
	\end{split}
	\end{equation}	
Invoking Young's inequality, we observe that
\begin{equation*}
	\int_\Omega e^{\tilde{C}\alpha(w+K)}w^2dx\leq e^{\tilde{C}K\alpha}\left(\int_\Omega e^{2\tilde{C}\alpha w}dx\right)^{1/2}\left(\int_\Omega w^4dx\right)^{1/2},
\end{equation*}
{ and thus
\begin{equation*}
	\int_\Omega\gamma^{-1}(v)w^2dx\leq
e^{\tilde{C}K\alpha}\left(\int_\Omega e^{2\tilde{C}\alpha w}dx\right)^{1/2}\left(\int_\Omega w^4dx\right)^{1/2}
+ 	\gamma^{-1}(s_b)\int_\Omega w^2dx.
\end{equation*}
 Here we apply Lemma \ref{lm2e} by taking $\|u_0\|_{L^1(\Omega)}=\Lambda$ and sufficiently small $\alpha>0$ such that $2\tilde{C}\alpha<A$ (recall that $\tilde{C}$ is independent of $\gamma$), 
and also invoke Lemma \ref{lm2} to have}
	\begin{equation*}
		\begin{split}
		\int_\Omega\gamma^{-1}(v)w^2dx\leq C.
		\end{split}
	\end{equation*}
Combining \eqref{wb00} with the above estimate completes the proof by solving the above differential inequality. 
\end{proof}
\begin{remark}
	If $\e=0$, we have
		\begin{equation}
	\sup\limits_{t\geq0}\left(\|\nabla v\|_{L^2(\Omega)}+\|v\|_{L^2(\Omega)}+\int_t^{t+1}\int_\Omega\gamma(v)u^2dxds\right)\leq C.
	\end{equation}
\end{remark}
\begin{remark}\label{rmchi1}
	If there is $0<\chi<\frac{4\pi}{\Lambda}$ with $\|u_0\|_{L^1(\Omega)}=\Lambda$ such that \begin{equation}
	\label{A2''}\mathrm{(A2'')}:\qquad
	\lim\limits_{s\rightarrow+\infty}e^{\chi s}\gamma(s)=+\infty,
	\end{equation}
we can argue in the same manner as before to deduce that
\begin{equation*}
\begin{split}
\int_\Omega \gamma^{-1}(v)w^2\leq b\int_\Omega e^{\tilde{C}\chi(w+K)}w^2+\int_\Omega\gamma^{-1}(s_b)w^2.
\end{split}
\end{equation*}
Since $\lim\limits_{s\rightarrow0}\gamma(s)=0$, by Lemma \ref{vbd}, $\tilde{C}>1$ above can be chosen arbitrarily close to $1$ such that 
	\begin{equation*}
		\tilde{C}\chi<\frac{4\pi}{\Lambda}.
	\end{equation*}
Moreover, we may fix some $p>1$ such that 	
		\begin{equation*}
	\tilde{C}p\chi<\frac{4\pi}{\Lambda}.
	\end{equation*}
	Thus, thanks to Young's inequality, Lemma \ref{lm2} and Lemma \ref{lm2e}, we infer that
	\begin{equation*}
		\begin{split}
		\int_\Omega e^{\tilde{C}\chi(w+K)}w^2\leq e^{\tilde{C}K\chi}\left(\int_\Omega e^{\tilde{C}p\chi w}\right)^{1/p}\left(\int_\Omega w^{2p'}\right)^{1/{2p'}}\leq C
		\end{split}
	\end{equation*}where $1/p+1/{p'}=1$. Thus, if $\Lambda<\frac{4\pi}{\chi}$, there also holds	\begin{equation}
	\sup\limits_{t\geq0}\left(\|\nabla w\|_{L^2(\Omega)}+\|w\|_{L^2(\Omega)}+\int_t^{t+1}\int_\Omega\gamma(v)u^2dxds\right)\leq C.
	\end{equation}

\end{remark}

\begin{lemma}
	Under the assumption of Thereom \ref{TH0}, we have
	\begin{equation}
	\sup\limits_{t\geq0}\left(\|w\|_{L^\infty(\Omega)}+\|v\|_{L^\infty(\Omega)}\right)\leq C.
	\end{equation}
\end{lemma}
\begin{proof}
	Recall that $w-\Delta w=u$. For any fixed $1<p<2$, we infer by the Sobolev embedding theorem and \eqref{cond_gamma0} that
	\begin{align*}
	\|w\|_{L^\infty(\Omega)}\leq &C\|u\|_{L^p(\Omega)}\\
	\leq&\left(\int_\Omega u^2\gamma(v)dx\right)^{\frac12}\left(\int_\Omega (\gamma(v))^{-\frac{p}{2-p}}dx\right)^{\frac{2-p}{2p}}\\
	\leq&C\left(\int_\Omega u^2\gamma(v)dx\right)^{\frac12}\left(\int_\Omega \left(be^{\tilde{C}\alpha (w+K)}+\gamma^{-1}(s_b)\right)^{\frac{p}{2-p}}dx\right)^{\frac{2-p}{2p}}.
	\end{align*}
Picking $\alpha>0$ small such that $\frac{\tilde{C}p\alpha}{2-p}<A$, we deduce by Lemma \ref{lm2e} that
\begin{equation}
	\|w\|_{L^\infty(\Omega)}\leq C\left(\int_\Omega u^2\gamma(v)dx\right)^{\frac12}.
\end{equation}	
	Then by Lemma \ref{wH1a}, for any $t>0$ we obtain that
	\begin{equation*}
	\int_t^{t+1}\|w\|_{L^\infty(\Omega)}ds\leq \int_t^{t+1}\int_\Omega u^2\gamma(v)dx+C\leq C
	\end{equation*}
	and thus for any fixed $x\in\Omega$,
	\begin{equation*}
	\sup\limits_{t>0}\int_t^{t+1}w(s,x)ds\leq C
	\end{equation*}
	with $C>0$ depending only on the initial data, $\gamma$ and $\Omega$. Finally, 
	observing that
	\begin{equation}\non
	w_t+u\gamma(v)=(I-\Delta)^{-1}[u\gamma(v)]\leq \gamma(v_*)(I-\Delta)^{-1}[u]=\gamma(v_*)w,
	\end{equation}
	we may apply the uniform Gronwall inequality Lemma \ref{uniformGronwall} to obtain that for any $x\in\Omega$
	\begin{equation*}
	w(x,t)\leq C\qquad\text{for}\;\;t\geq1,
	\end{equation*}with some $C>0$ independent of $x\in\Omega$, 
{ 	which together with Lemma \ref{keylem1} for $t\leq1$ gives rise to the following estimate} 
	\begin{equation*}
	w(x,t)\leq C\qquad\text{for}\;\;t\geq0.
	\end{equation*}
Finally, recall Lemma \ref{vbd}, we also have
	\begin{equation*}
	v(x,t)\leq C(w(x,t)+1)\leq C\qquad\text{for}\;\;t\geq0,
	\end{equation*}which  concludes the proof.
\end{proof}
\begin{remark}
	Under the same assumption of Remark \ref{rmchi1}, we can choose some $1<p<2$ such that $\frac{p}{2-p}$ larger than $1$ but  sufficiently close to $1$ such that
	\begin{equation*}
		\frac{\tilde{C}p\chi}{2-p}<\frac{4\pi}{\Lambda}.
	\end{equation*}Then, in the same manner as before, we can still deduce that 
	\begin{equation*}
	\begin{split}
	\|w\|_{L^\infty(\Omega)}\leq& C\|u\|_{L^p(\Omega)}\\
	\leq&\left(\int_\Omega u^2\gamma(v)dx\right)^{\frac12}\left(\int_\Omega (\gamma(v))^{-\frac{p}{2-p}}dx\right)^{\frac{2-p}{2p}}\\
	\leq&C\left(\int_\Omega u^2\gamma(v)dx\right)^{\frac12}\left(\int_\Omega \left(be^{\tilde{C}\chi (w+K)}+\gamma^{-1}(s_b)\right)^{\frac{p}{2-p}}dx\right)^{\frac{2-p}{2p}}\\
	\leq& C\left(\int_\Omega u^2\gamma(v)dx\right)^{\frac12}.
	\end{split}
	\end{equation*}Then, we can similarly prove that 
	\begin{equation*}
		\sup\limits_{t\geq0}\left(\|w\|_{L^\infty(\Omega)}+\|v\|_{L^\infty(\Omega)}\right)\leq C.
	\end{equation*}
\end{remark}
Once $v$ is uniformly-in-time bounded from above, we can prove in the same manner as in \cite{FJ19a,FJ19b} to get the uniform boundedness of the classical solutions and thus Theorem \ref{TH0} is proved. We omit the detail here.

\subsection{Higher-dimensional cases}
In this part, we aim to establish uniform-in-time upper bound for $v$ in higher dimensions when $\gamma$ decreases algebraically at large concentrations.
\begin{proposition}\label{Propunbv}
Assume $n\geq3$. Suppose $\gamma$ satisfies $\mathrm{(A0)}$ and $\mathrm{(A2)}$ with some $0<k<\frac{2}{n-2}$ when $\e=0$, and $\gamma$ satisfies  $(\mathrm{A1'})$ additionally  when $\e>0$. Then there is $C>0$ depending only on the initial data, $\gamma$, { $\e$} and $\Omega$ such that
\begin{equation}
	\sup\limits_{0\leq t<T_{\mathrm{max}}} \|v(\cdot,t)\|_{L^\infty(\Omega)}\leq C.
\end{equation}
\end{proposition}
The proof of the above result consists of several steps. To begin with, we prove the following time-independent estimates.
\begin{lemma}\label{wH1b}
Under the same assumptions of Proposition \ref{Propunbv}, there is $C>0$ depending only on the initial data, $\gamma$, { $\e$} and $\Omega$ such that
\begin{equation}
	\sup\limits_{0\leq t<T_{\mathrm{max}}}\left(\|\nabla w\|_{L^2(\Omega)}+\|w\|_{L^2(\Omega)}\right)\leq C,
\end{equation}and for any $t\in(0,T_{\mathrm{max}}-\tau)$ with $\tau=\min\{1,\frac12T_{\mathrm{max}}\}$,
\begin{equation}
	\int_t^{t+\tau}\int_\Omega\gamma(v)u^2dxds\leq C.
\end{equation}
\end{lemma}
	\begin{proof}
 Proceeding the same lines as in Lemma \ref{wH1a}, we arrive at  
	\begin{equation}\label{wb0}
	\begin{split}
	&\frac{d}{dt}\left(\|\nabla w\|_{L^2(\Omega)}^2+\|w\|_{L^2(\Omega)}^2\right)
+\|\nabla w\|_{L^2(\Omega)}^2+\|w\|_{L^2(\Omega)}^2
	+\int_\Omega \gamma(v)u^2dx\\
	&\leq C\int_\Omega\gamma^{-1}(v)w^2dx,
	\end{split}
	\end{equation} 
	and we will estimate the right-hand side.
Under the assumption of Proposition \ref{Propunbv}, we may infer that there exist  $k \in (0, \frac{2}{n-2})$, $b>0$ and $s_b>v_*$ such that for all $s\geq s_b$
\begin{equation*}
\gamma^{-1}(s)\leq bs^k
\end{equation*}and on the other hand, since $\gamma(\cdot)$ is decreasing,
\begin{equation*}
\gamma^{-1}(s)\leq \gamma^{-1}(s_b)
\end{equation*}for all $0\leq s<s_b$.
Therefore, for all $s\geq0$, there holds
\begin{equation}\label{cond_gamma}
\gamma^{-1}(s)\leq bs^{k}+\gamma^{-1}(s_b).
\end{equation}
Thus, we deduce from above and Lemma \ref{vbd}
that	\begin{equation*}
\begin{split}
\int_\Omega\gamma^{-1}(v)w^2dx\leq &\int_\Omega (bv^k+\gamma^{-1}(s_b))w^2dx\\
\leq&\int_\Omega \left(b\left(C(w+1)\right)^k+\gamma^{-1}(s_b)\right)w^2dx\\
\leq &C\int_\Omega w^{k+2}dx+C
\end{split}
\end{equation*}
with $C>0$ depending only on the initial data, $\gamma$, {$\e$} and $\Omega$.

Recall that $\|w\|_{L^q(\Omega)}$ with  any $q\in[1,\frac{n}{n-2})$ is bounded due to Lemma \ref{lm2}. Thus if $k+2< \frac{n}{n-2}$, which only occurs when $n=3$ and $0<k<1$, there holds
\begin{equation*}
\int_\Omega w^{k+2}dx\leq C.
\end{equation*}
On the other hand, if $1\leq k<2$ when $n=3$ or $0<k<\frac{2}{n-2}$ when $n\geq4$, 
we can check 
$$\frac{n}{n-2}\leq k+2\leq  q_*,\qquad
\frac{nk}{2}<\frac{n}{n-2}$$
with $q_*\triangleq\frac{2n}{n-2}$. 
Here we can pick up $q\geq 1$ satisfying
$$ \frac{nk}{2}  < q < \frac{n}{n-2}, $$ 
and make use of the interpolation inequality and the Sobolev embedding $H^1\hookrightarrow L^{q_*}$ to have  
\begin{equation*}
\int_\Omega w^{k+2}dx\leq \|w\|_{L^{q_*}(\Omega)}^{\beta_1(k+2)}\|w\|_{L^q(\Omega)}^{(1-\beta_1)(k+2)}\leq C\| w\|_{H^1(\Omega)}^{\beta_1(k+2)}
\end{equation*}with some $C>0$ depending only on $n,\Omega$ and the initial data, and
\begin{equation*}
\beta_1=(\frac{1}{q}-\frac{1}{k+2})/(\frac1{q}-\frac{1}{q_*}).
\end{equation*}
Moreover, we can easily confirm that
\begin{equation*}
0\leq\beta_1(k+2)<2
\end{equation*}
due to $\frac{nk}{2}  < q $. 
By invoking Young's inequality we arrive at 
\begin{equation}\label{wb01}
\int_\Omega\gamma^{-1}(v)w^2dx
\leq 
C\int_\Omega w^{k+2}dx+C
\leq \frac{1}{2}\| w\|_{H^1(\Omega)}^{2}+C.
\end{equation}

In summary, for $n\geq3$ and $0<k<\frac{2}{n-2}$, 
\eqref{wb0} and \eqref{wb01} implies
\begin{equation}
\frac{d}{dt}(\|\nabla w\|_{L^2(\Omega)}^2+\|w\|_{L^2(\Omega)}^2)+\frac12(\|\nabla w\|_{L^2(\Omega)}^2+\|w\|_{L^2(\Omega)}^2)+\frac12\int_\Omega \gamma(v)u^2dx\leq C,
\end{equation}
where $C>0$ depends only on the initial data, $\gamma$,  $\e$ and $\Omega$. Then a direct ODE analysis will finally yield to our assertion. This completes the proof.
	\end{proof}
With the above result, we can establish the uniform-in-time upper bounds of $v$. For comparison, we first provide a simple proof in the same spirit as given in Sect. 3.1 which relies on an application of the uniform Gronwall inequality and the three-dimensional embeddings.
\begin{lemma}
	Assume $n=3$ and suppose $\gamma$ satisfies $\mathrm{(A0)}$ and $\mathrm{(A2)}$ with some  $0<k<2$. Moreover, $\gamma$ satisfies  $(\mathrm{A1'})$ additionally when $\e>0.$ There is $C>0$ depending only on $\Omega$, $k,\e$ and the initial data such that
\begin{equation}
	\sup\limits_{0\leq t<T_{\mathrm{max}}}\left(\|w\|_{L^\infty(\Omega)}+\|v\|_{L^\infty(\Omega)}\right)\leq C.
\end{equation}
\end{lemma}
\begin{proof}
For any $\frac{3}{2}<p<2$, due to the three-dimensional Sobolev embedding theorem and H\"older's inequality, we have
	\begin{equation*}
	\begin{split}
	\|w\|_{L^\infty(\Omega)}\leq&  C\|u\|_{L^p(\Omega)}\\
	\leq& C\left(\int_\Omega \gamma(v)u^2dx\right)^{1/2}\left(\int_\Omega (\gamma(v))^{-\frac{p}{2-p}}dx\right)^{\frac{2-p}{2p}}.
	\end{split}
	\end{equation*}	
	In the same manner as before, we infer that
\begin{equation}
\begin{split}
\int_\Omega (\gamma(v))^{-\frac{p}{2-p}}dx\leq & \int_\Omega \left(b v^{k}+\gamma^{-1}(s_b)\right)^{\frac{p}{2-p}}dx\\
\leq&\int_\Omega \left(b\left(C(w+1)\right)^k+\gamma^{-1}(s_b)\right)^{\frac{p}{2-p}}dx\\
\leq &C\int_\Omega w^{\frac{pk}{2-p}}dx+C,
\end{split}
\end{equation} where $C>0$ depends only on the initial data, $\gamma$,  $\e$ and $\Omega$.

Since $0<k<2$, we can always pick  $\frac32<p<2$ such that  $\frac{pk}{2-p}\leq 6$ and hence by the three-dimensional Sobolev embeddings  and Lemma \ref{wH1b},
\begin{equation*}
\left(\int_\Omega (\gamma(v))^{-\frac{p}{2-p}}dx\right)^{\frac{2-p}{2p}}
\leq
C
\left(\int_\Omega w^{\frac{pk}{2-p}}dx\right)^{\frac{2-p}{2p}}+C
\leq C\| w\|_{H^1(\Omega)}^{\frac{k}{2}}+C\leq C.
\end{equation*}
As a result, invoking Lemma \ref{wH1b} again, for any $t\in(0,T_{\mathrm{max}}-\tau)$ with  $\tau=\min\{1,\frac12T_{\mathrm{max}}\}$,
	\begin{equation}\label{ubw2}
	\int_t^{t+\tau}\|w\|_{L^\infty(\Omega)}\leq C \int_t^{t+\tau}\int_\Omega\gamma(v)u^2dxds+C\leq C.
	\end{equation}
It follows that for any fixed $x\in\Omega$ and any $t\in(0,T_{\mathrm{max}}-\tau)$ with  $\tau=\min\{1,\frac12T_{\mathrm{max}}\}$,
	\begin{equation}
	\int_t^{t+\tau}w(x,s)ds\leq	\int_t^{t+\tau}\|w\|_{L^\infty(\Omega)}\leq C.
	\end{equation}
	Then, we recall the key identity \eqref{keyid} and deduce by the comparison principle of elliptic equations that
	\begin{equation*}
	w_t+\gamma(v)u=(I-\Delta)^{-1}[\gamma(v)u]\leq (I-\Delta)^{-1}[\gamma(v_*)u]=\gamma(v_*)w.
	\end{equation*}
	Since $u\gamma(v)\geq0$, with the aid of the uniform Gronwall inequality (Lemma \ref{uniformGronwall}), we infer for any $x\in\Omega$ and $t\in(\tau,T_{\mathrm{max}})$ that
	\begin{equation}\non
	w(x,t)\leq C
	\end{equation}with some $C>0$ independent of $x$, $t$ and $T_{\mathrm{max}}$ which together with Lemma \ref{keylem1} for $t\leq \tau$ gives rise to the uniform-in-time boundedness of $w$ such that for all $(x,t)\in\Omega\times[0,T_{\mathrm{max}}),$
	\begin{equation}\non
	w(x,t)\leq C.
	\end{equation}
	This completes the proof  in view of  Lemma \ref{vbd}.
\end{proof}
For higher dimensions, the preceding argument fails. We provide the following alternative proof which is based on a  modified Alikakos--Moser iteration \cite{Alik79}. First, we begin with the case $\e=0$  and keep in mind that in such case $w$ is identical to $v$. 
\begin{lemma}\label{lmvbd0}
	Assume that $n\geq3$ and $\e=0$. Suppose $\gamma$ satisfies $\mathrm{(A0)}$ and $\mathrm{(A2)}$ with some $0<k<\frac{2}{n-2}$.  
There is $C>0$ depending only on $\Omega$, $k$ and the initial data such that
	\begin{equation}
	\sup\limits_{0\leq t<T_{\mathrm{max}}}\|v\|_{L^\infty(\Omega)}\leq C.
	\end{equation}
\end{lemma}

 We prepare the following auxiliary lemmas. 

\begin{lemma}\label{lmvbd0pre1}
	Assume that $n\geq3$ and $\e=0$. Suppose $\gamma$ satisfies $\mathrm{(A0)}$ and $\mathrm{(A2)}$ with some $0<k<\frac{2}{n-2}$.  
	There exist some $\lambda_1>0$ and $\lambda_2>0$
	such that for any  $p>\frac{n}{n-2}$, 
	\begin{equation}
	\frac{d}{dt}\int_\Omega v^{p}
+\lambda_2 p\int_\Omega v^p	
	+\frac{\lambda_1 p(p-k-1)}{(p-k)^2}\int_\Omega|\nabla v^{\frac{p-k}{2}}|^2+\lambda_1 p\int_\Omega v^{p-k}
	\leq  2\lambda_2 p\int_\Omega v^p.
	\end{equation}
\end{lemma}

\begin{proof}
Let $p>\frac{q_*}{2}$ with $q_*=\frac{2n}{n-2}$.
	Multiplying the key identity \eqref{keyid} by $v^{p-1}$, we obtain that
	\begin{equation}
	\frac{1}{p}\frac{d}{dt}\int_\Omega v^p+\int_\Omega u\gamma(v)v^{p-1}=\int_\Omega (I-\Delta)^{-1}[u\gamma(v)]v^{p-1}\leq\gamma(v_*)\int_\Omega v^p.
	\end{equation}	
	Thanks to \eqref{cond_gamma}, it follows   that
	\begin{equation*}
	\begin{split}
	\int_\Omega u\gamma(v)v^{p-1}dx\geq&\int_\Omega u \bigg(bv^{k}+\gamma^{-1}(s_b)\bigg)^{-1}v^{p-1}dx\\
	\geq&C\int_\Omega (v^{k}+1)^{-1}v^{p-1}udx
	\end{split}
	\end{equation*}with $C>0$ independent of $p$ and time.
	Since $v^k\geq v_*^k$ by Lemma \ref{lbdv}, there holds
	\begin{equation}
		\begin{split}
		(v^k+1)^{-1}v^{p-1}\geq(v^k+v_*^{-k}v^k)^{-1}v^{p-1}=\frac{v^{p-k-1}}{1+v_*^{-k}}
		\end{split}
	\end{equation}
from which we deduce that	
		\begin{equation}\label{gammab}
	\begin{split}
	\int_\Omega u\gamma(v)v^{p-1}dx\geq C\int_\Omega v^{p-k-1}udx
	\end{split}
	\end{equation}where $C>0$ depends on the initial data, $\Omega$ and $\gamma$, but is  independent of $p$ and time.
	
Next, recalling that $v-\Delta v=u$, we observe that
	\begin{equation*}
	\begin{split}
	\int_\Omega v^{p-k-1}udx=&\int_\Omega v^{p-k-1}(v-\Delta v)dx\\
	=&\int_\Omega v^{p-k}dx+(p-k-1)\int_\Omega|\nabla v|^2v^{p-k-2}\\
	=&\int_\Omega v^{p-k}dx+\frac{4(p-k-1)}{(p-k)^2}\int_\Omega|\nabla v^{\frac{p-k}{2}}|^2.
	\end{split}
	\end{equation*}
	Therefore, we arrive at
	\begin{equation}
	\frac{d}{dt}\int_\Omega v^{p}+\frac{\lambda_1 p(p-k-1)}{(p-k)^2}\int_\Omega|\nabla v^{\frac{p-k}{2}}|^2+\lambda_1 p\int_\Omega v^{p-k}\leq  \lambda_2 p\int_\Omega v^p
	\end{equation}with some $\lambda_1,\lambda_2>0$ independent of $p$ and time. 
	Adding $\lambda_2 p\int_\Omega v^p$, we complete the proof.
\end{proof}

\begin{lemma}\label{lmvbd0pre2}
	Assume that $n\geq3$ and $\e=0$. Suppose $\gamma$ satisfies $\mathrm{(A0)}$ and $\mathrm{(A2)}$ with some $0<k<\frac{2}{n-2}$.  
{ Let $L>1$.} 
There exists $C_0>0$ depending only on the initial data, $\Omega,k$ and $n$ such that for any  $p >q\geq\frac{n}{n-2}$ satisfying
$$
q<p=2q-\frac{nk}{2},
$$ 
there holds
\begin{equation}\non
\frac{d}{dt}\int_\Omega v^{p}+\lambda_2p\int_\Omega v^p\leq  C_0L^{\frac{n}{2}}p^{\frac{n+2}{2}}\left(\int_\Omega v^q\right)^2.
\end{equation}
\end{lemma}
	
\begin{proof}
Let $p>q\geq\frac{q_*}{2}$ satisfying 
$$
q<p=2q-\frac{nk}{2} = 2q-\frac{kq_*}{q_*-2}.
$$ 
Denote $\eta=v^{\frac{p-k}{2}}$ and  define
\begin{equation}\label{alpha}
\alpha=\frac{(p-k)(p-q)}{p(p-k-2q/q_*)}.
\end{equation}	
One easily checks that $\alpha\in(0,1)$. Indeed, 
\begin{equation*}
	\begin{split}
	p-k-\frac{2q}{q_*}>&q-\frac{2q}{q_*}-k
=\frac{q_*-2}{q_*}q-k\\
\geq&	\frac{q_*-2}{q_*}\frac{q_*}{2}-k
=\frac{2}{n-2}-k>0
	\end{split}
\end{equation*}
and on the other hand, solving $\alpha<1$ yields $p>\frac{kq_*}{q_*-2}$, which is guaranteed by $p>q_*/2$ since $\frac{q_*}{2}>\frac{kq_*}{q_*-2}$ and $k<\frac{2}{n-2}$. Moreover, since $k<\frac{2}{n-2}$,  there holds $\frac{2p\alpha}{p-k}<2$ provided that $q\geq\frac{q_*}{2}$.  Then an application of H\"older's inequality yields that
	\begin{equation*}
	\begin{split}
		\int_\Omega v^{p}dx&=\int_\Omega \eta^{\frac{2p}{p-k}}dx=\int_\Omega\eta^{\frac{2p\alpha}{p-k}}\eta^{\frac{2p(1-\alpha)}{p-k}}dx\\
		&\leq\left(\int_\Omega \eta^{q_*}dx\right)^{\frac{2p\alpha}{(p-k)q_*}}\left(\int_\Omega \eta^{\frac{2p(1-\alpha)q_*}{(p-k)q_*-2p\alpha}} \right)^{\frac{(p-k)q_*-2p\alpha}{(p-k)q_*}}\qquad(\text{since}\;\;\frac{2p\alpha}{p-k}<2<q_*)\\
		&=\| \eta\|^{\frac{2p\alpha}{p-k}}_{L^{q_*}(\Omega)}\left(\int_\Omega \eta^{\frac{2q}{p-k}}\right)^{\frac{(p-k)q_*-2p\alpha}{(p-k)q_*}}\\
			&=\| \eta\|^{\frac{2p\alpha}{p-k}}_{L^{q_*}(\Omega)}\left(\int_\Omega v^q\right)^{\frac{(p-k)q_*-2p\alpha}{(p-k)q_*}}.
	\end{split}
	\end{equation*}
Recall the Sobolev embedding inequality 
\begin{equation*}
\|\eta\|_{L^{q_*}(\Omega)}\leq \lambda_*\|\eta\|_{H^1(\Omega)}
\end{equation*}where $q_*=\frac{2n}{n-2}$ and $\lambda_*>0$ depends only on $n$ and $\Omega$. In view of the fact $\frac{2p\alpha}{p-k}<2$, invoking Young's inequality, we obtain that
\begin{equation*}
	\begin{split}
		&\lambda_2p\int_\Omega v^{p}dx\\
	\leq&\lambda_2p\| \eta\|^{\frac{2p\alpha}{p-k}}_{L^{q_*}(\Omega)}\left(\int_\Omega v^{q}\right)^{\frac{(p-k)q_*-2p\alpha}{(p-k)q_*}}\\
	\leq&\lambda_2p\bigg(\lambda_*\| \eta\|_{H^1(\Omega)}\bigg)^{\frac{2p\alpha}{p-k}}\left(\int_\Omega v^{q}\right)^{\frac{(p-k)q_*-2p\alpha}{(p-k)q_*}}\\
	\leq&\frac{p\alpha\delta^{\frac{p-k}{p\alpha}}}{p-k}\|\eta\|^2_{H^1(\Omega)}+\frac{p-k-p\alpha}{p-k}\lambda_*^{\frac{2p\alpha}{p-k-p\alpha}}\left(\delta^{-1}\lambda_2p\right)^{\frac{p-k}{p-k-p\alpha}}\left(\int_\Omega v^{q}\right)^{\frac{(p-k)q_*-2p\alpha}{(p-k-p\alpha)q_*}},
	\end{split}
\end{equation*}	
where $\delta>0$ satisfies
	\begin{equation}\label{choise_delta}
		\frac{p\alpha\delta^{\frac{p-k}{p\alpha}}}{p-k}=\frac{\lambda_1p(p-k-1)}{2L(p-k)^2}.
	\end{equation} 
	It follows from above and \eqref{alpha} that
\begin{equation}
	\begin{split}
	&\frac{p-k-p\alpha}{p-k}\lambda_*^{\frac{2p\alpha}{p-k-p\alpha}}\left(\delta^{-1}\lambda_2p\right)^{\frac{p-k}{p-k-p\alpha}}\left(\int_\Omega v^{q}\right)^{\frac{(p-k)q_*-2p\alpha}{(p-k-p\alpha)q_*}}\\
	=&\frac{p-k-p\alpha}{p-k}\left(\frac{2L\alpha(p-k)\lambda_*^2}{\lambda_1(p-k-1)}\right)^{\frac{p\alpha}{p-k-p\alpha}}\left(\lambda_2 p\right)^{\frac{p-k}{p-k-p\alpha}}\left(\int_\Omega v^{q}\right)^{\frac{(p-k)q_*-2p\alpha}{(p-k-p\alpha)q_*}}\\
	=&\frac{(q_*-2)q-kq_*}{q_*(p-k)-2q}\left(\frac{2L(p-k)^2(p-q)\lambda_*^2}{\lambda_1 p(p-k-1)(p-k-2q/q_*)}\right)^{\frac{(p-q)q_*}{q(q_*-2)-kq_*}}(\lambda_2p)^{\frac{q_*(p-k)-2q}{(q_*-2)q-kq_*}}\\
	&\times\left(\int_\Omega v^q\right)^{\frac{q_*(p-k)-2p}{q_*(q-k)-2q}}.
	\end{split}
\end{equation}	

Since $p>q\geq\frac{q_*}{2}$ satisfying 
$$
q<p=2q-\frac{nk}{2} = 2q-\frac{kq_*}{q_*-2},
$$
one easily checks that
 \begin{eqnarray*}
 \frac{q_*(p-k)-2p}{q_*(q-k)-2q}&=&2,\\
\frac{(q_*-2)q-kq_*}{q_*(p-k)-2q}&=&\frac{2}{n+2},\\
\frac{(p-q)q_*}{q(q_*-2)-kq_*}&=&\frac{n}{2},\\
\frac{q_*(p-k)-2q}{(q_*-2)q-kq_*}&=&\frac{n+2}{2},
\end{eqnarray*}
 and
 \begin{equation*}
 	\frac{p-q}{p-k-2q/q_*}=\frac{n}{n+2}.
 \end{equation*}
Moreover, since $p>\frac{q_*}{2}>1$,
\begin{equation}\non
	\begin{split}
\frac{(p-k)^2}{p(p-k-1)}=&\frac{(p-k-1)^2+2(p-k-1)+1}{p(p-k-1)}\\
=&\frac{p-k-1}{p}+\frac{2}{p}+\frac{1}{p(p-k-1)}\\
<&3+\frac{1}{p-k-1}\\
<&3+\frac{1}{\frac{2}{n-2}-k},
	\end{split}
\end{equation}and
\begin{equation}\label{llll}
	\frac{(p-k)^2}{p(p-k-1)}>\frac{p-k}{p}=1-\frac{k}{p}>1-\frac{2k}{q_*}=1-\frac{(n-2)k}{n}>0.
\end{equation}
Therefore by the above calculations, it follows
\begin{eqnarray*}
\frac{2L(p-k)^2(p-q)\lambda_*^2}{\lambda_1 p(p-k-1)(p-k-2q/q_*)}
&=&\frac{2L\lambda_*^2}{\lambda_1}
\cdot \frac{n}{n+2}
\cdot
\frac{(p-k)^2}{p(p-k-1)}\\
&<&\frac{2Ln\lambda_*^2}{\lambda_1(n+2)}
\left(3+\frac{1}{\frac{2}{n-2}-k}\right).
\end{eqnarray*}
Hence one can find   $C_0>0$ being a constant depending only on the initial data, $\Omega,k$ and $n$ such that
\begin{eqnarray*}
&&\frac{(q_*-2)q-kq_*}{q_*(p-k)-2q}
\left(\frac{2L(p-k)^2(p-q)\lambda_*^2}{\lambda_1 p(p-k-1)(p-k-2q/q_*)}\right)^{\frac{(p-q)q_*}{q(q_*-2)-kq_*}}
(\lambda_2p)^{\frac{q_*(p-k)-2q}{(q_*-2)q-kq_*}}\\
&<&\frac{2}{n+2} \cdot
\left\{\frac{2Ln\lambda_*^2}{\lambda_1(n+2)}
\left(3+\frac{1}{\frac{2}{n-2}-k}\right)
 \right\}^{\frac{n}{2}}
(\lambda_2p)^{\frac{n+2}{2}} \\
	&\leq& \frac{C_0}{2}L^{\frac{n}{2}}p^{\frac{n+2}{2}}.
\end{eqnarray*}
Therefore by the above and \eqref{choise_delta} we have
\begin{eqnarray*}
		2\lambda_2p\int_\Omega v^{p}dx
	\leq
\frac{\lambda_1p(p-k-1)}{L(p-k)^2}
	\|v^{\frac{p-k}{2}}\|^2_{H^1(\Omega)}
	+ C_0L^{\frac{n}{2}}p^{\frac{n+2}{2}}\left(\int_\Omega v^{q}\right)^{2}.
\end{eqnarray*}	
Combining Lemma \ref{lmvbd0pre1} and recalling $L>1$,
 we obtain the following inequality
\begin{equation}\non
\frac{d}{dt}\int_\Omega v^{p}+\lambda_2p\int_\Omega v^p\leq  C_0L^{\frac{n}{2}}p^{\frac{n+2}{2}}\left(\int_\Omega v^q\right)^2.
\end{equation}
\end{proof}

Now we are in a position to give a proof of Lemma \ref{lmvbd0}.
\begin{proof}
For all $r\in\mathbb{N}$ we define 
$$p_r\triangleq2^{r-1}(q_*-nk)+\frac{nk}{2},\qquad p_0=q_*/2.$$ 
Then $p_r> \frac{n}{n-2}$ and $p_r=2p_{r-1}-\frac{nk}{2}$. 
We apply Lemma \ref{lmvbd0pre2} with $(p,q)=(p_r, p_{r-1})$ to have 
\begin{equation}\non
\frac{d}{dt}\int_\Omega v^{p_r}+ \lambda_2 p_r\int_\Omega v^{p_r}\leq  \lambda_2p_r\mathcal{A}_r
\left(\mathcal{M}_{r-1}  \right)^{2},
\end{equation}
where
\begin{equation*}
\mathcal{M}_{r}\triangleq\sup\limits_{0\leq t< T_{\mathrm{max}}}\int_\Omega v^{p_r}
\quad
\mbox{and}
\quad
	\mathcal{A}_r\triangleq\frac{C_0L^{\frac{n}{2}}p_r^\frac{n}{2}}{\lambda_2}.
\end{equation*}
By solving the above ODE, it follows that for all $r\in\mathbb{N}$
\begin{equation*}
\mathcal{M}_{r} = \sup\limits_{0\leq t< T_{\mathrm{max}}}\int_\Omega v^{p_r}\leq\max\{\mathcal{A}_r\mathcal{M}_{r-1}^2,\| v_0\|_{L^\infty(\Omega)}^{p_r}   \}.
\end{equation*}
Since $p_r\geq q_*/2$ for all $r\geq1$, one can choose $L>1$ sufficiently large depending only on the initial data, $\Omega$, $n$ and $k$ such that $\mathcal{A}_r>1$ for all $r\geq1.$ Moreover,  adjusting $C_0$ by a proper larger number, we have
\begin{equation}
	\begin{split}\non
	\mathcal{A}_r	\leq C_0a^{r}
	\end{split}
\end{equation}with some $a>0$ depending only on the initial data, $\Omega,$ $k$ and $n$. 
In addition, due to Lemma \ref{wH1b}, we may find some large constant $K_0>1$  that dominates $\|v_0\|_{L^\infty}$ and $\int_\Omega v^{q_*/2}$ for all time. 

Iteratively, we deduce that
\begin{equation}
	\begin{split}
	\int_\Omega v^{p_r}\leq&\max\{\mathcal{A}_r\mathcal{A}_{r-1}^2\mathcal{M}^4_{r-2},\mathcal{A}_rK_0^{2p_{r-1}},K_0^{p_r}\}\\
	=&\max\{\mathcal{A}_r\mathcal{A}_{r-1}^2\mathcal{M}^4_{r-2},\mathcal{A}_rK_0^{2p_{r-1}}\}\\
	\leq&\dots\\
	\leq&\max\{\mathcal{A}_r\mathcal{A}_{r-1}^2\mathcal{A}_{r-2}^4\cdots\mathcal{A}_1^{2^{r-1}}\mathcal{M}_0^{2^r},\mathcal{A}_r\mathcal{A}_{r-1}^2\cdots\mathcal{A}_2^{2^{r-2}}K_0^{2^{r-1}p_1}\} \\
	\leq&\max\{\mathcal{A}_r\mathcal{A}_{r-1}^2\mathcal{A}_{r-2}^4\cdots\mathcal{A}_1^{2^{r-1}}K_0^{2^r},\mathcal{A}_r\mathcal{A}_{r-1}^2\cdots\mathcal{A}_2^{2^{r-2}}K_0^{2^{r-1}p_1}\}\\
	\leq& C_0^{2^0+2^1+\cdots+2^{r-1}}
	\times a^{1\cdot r+2(r-1)+2^2(r-2)+\cdots+2^{r-1}(r-(r-1))}\times\tilde{K}_0^{2^r}\\
	=&{ C_0^{2^r-1}} a^{2^{1+r}-r-2}\tilde{K}^{2^r}_0\non
	\end{split}
\end{equation}where $\tilde{K}=\max\{K_0,K_0^{\frac{p_1}{2}}\}$.
Finally, recalling that $p_r= 2^{r-1}(q_*-nk)+\frac{nk}{2}$,  we deduce that
\begin{equation}\non
	\|v\|_{L^\infty(\Omega)}\leq\lim\limits_{r\nearrow+\infty}\left({ C_0^{2^r-1} } a^{2^{1+r}-r-2}\tilde{K}^{2^r}_0\right)^{1/p_r}=\left(C_0a^2\tilde{K}_0\right)^{\frac{2}{q_*-nk}},
\end{equation}
which concludes the proof.
\end{proof}

Next, we turn to consider the fully parabolic case $\e>0$. 
{ Without loss of generality, we assume $\e=1$. 
For any $\e>0$ we can proceed the same lines to obtain the following lemma.}
\begin{lemma}\label{lmvbd1}
	Assume that $n\geq3$, $\e=1$.  Suppose that $\gamma$ satisfies $\mathrm{(A0)}$, $\mathrm{(A1')}$ and $\mathrm{(A2)}$ with some $k<\frac{2}{n-2}$. There is $C>0$ depending only on $\Omega$, $k,\e$ and the initial data such that
\begin{equation}\non
\sup\limits_{0\leq t<T_{\mathrm{max}}}\left(\|w\|_{L^\infty(\Omega)}+\|v\|_{L^\infty(\Omega)}\right)\leq C.
\end{equation}
\end{lemma}
\begin{proof}
First of all, we note that since  $w\geq w_*>0$ by Lemma \ref{lbdw}. It follows from Lemma \ref{vbd} that
\begin{equation}\non
	v\leq C(w+1)\leq C(w+\frac{w}{w_*})=C(1+\frac1{w_*})w
\end{equation}with some $C>0$ depending only on $\Omega,\gamma,\e$ and the initial data. Hence by the non-increasing property of $\gamma$,
\begin{equation*}
	\begin{split}
		\gamma(v)\geq\gamma (C(w+1))\geq\gamma(C'w)
	\end{split}
\end{equation*}with $C'=C(1+\frac{1}{w_*})$.
 Now denoting $\w=C'w$, it follows  from \eqref{keyid} that
 \begin{equation}	 	\label{keyid2}
\w_t+C'u\gamma(\w)\leq C'(I-\Delta)^{-1}[u\gamma(v)]\leq \gamma(v_*)\w.
 \end{equation}
 Now, multiplying \eqref{keyid2} by $\w^{p-1}$, we get
\begin{equation}\non
\frac{1}{p}\frac{d}{dt}\int_\Omega \w^pdx+C'\int_\Omega u\gamma(\w)\w^{p-1}dx\leq \gamma(v_*)\int_\Omega \w^pdx.
\end{equation}	
Here, we note that $\tilde{w}-\Delta\tilde{w}=C'u.$ Then in the same manner as done in proof of Lemma \ref{lmvbd0}, one proves that there is $C>0$ depending only on the initial data, $\Omega,\e$ and $k$ such that
\begin{equation}\non
	\sup\limits_{0\leq t<T_{\mathrm{max}}}\|\tilde{w}\|_{L^\infty(\Omega)}\leq C.
\end{equation}
This completes the proof since $v\leq \tilde{w}$ point-wisely.
\end{proof}

Before concluding this section, we show the relationship between  $\mathrm{(A2)}$ and $\mathrm{(A3)}$-type conditions  provided that assumptions $\mathrm{(A0)}$ and $\mathrm{(A1)}$ are satisfied.
\begin{lemma}\label{lemA23}
	A function satisfying $\mathrm{(A0)}$, $\mathrm{(A1)}$ and  
\begin{equation}\label{A3c}
\mathrm{(A3c)}:\qquad l|\gamma'(s)|^2\leq \gamma(s)\gamma''(s),\;\;\forall\;s>0
\end{equation}
with some $l>1$ must fulfill assumption  $\mathrm{(A2)}$ with any $k>\frac{1}{l-1}$.
\end{lemma}
\begin{proof}
	First, we point out that under the assumptions $\mathrm{(A0)}$, $\mathrm{(A1)}$ and $\mathrm{(A3c)}$, $\gamma'(s)<0$ on $[0,\infty).$ In fact, due to $\mathrm{(A0)}$ and $\mathrm{(A3c)}$, we have $\gamma''(s)\geq0$ for all $s>0$. Then if there is $s_1\geq0$ such that $\gamma'(s_1)=0$, it must hold that $0=\gamma'(s_1)\leq \gamma'(s)\leq0$ for all $s\geq s_1$, which contradicts to the positivity of $\gamma$ in assumptions $\mathrm{(A0)}$  and the asymptotically vanishing assumption $\mathrm{(A1)}$.
	
	Now, we may divide \eqref{A3c} by $-\gamma(s)\gamma'(s)$ to obtain that
	\begin{equation*}
	-\frac{l\gamma'(s)}{\gamma(s)}\leq -\frac{\gamma''(s)}{\gamma'(s)},\;\;\;\;\forall s>0,
	\end{equation*}
	which indicates that
	\begin{equation*}
	\left(\log(-\gamma^{-l}\gamma')\right)'\leq0.
	\end{equation*}
	An integration of above ODI from $v_*$ to $s$ yields that
	\begin{equation}\non
	-\gamma^{-l}(s)\gamma'(s)\leq-\gamma^{-l}(v_*)\gamma'(v_*)\triangleq d>0,
	\end{equation}which further implies that
	\begin{equation*}
	\left(\frac{1}{(l-1)\gamma^{l-1}(s)}\right)'\leq d.
	\end{equation*}
	Thus for any $s\geq v_*$, there holds
	\begin{equation*}
	\frac{1}{\gamma^{l-1}(s)}\leq d(l-1)(s-v_*)+\frac{1}{\gamma^{l-1}(v_*)}.
	\end{equation*}
	As a result, for any $k>\frac{1}{l-1}$, we have
	\begin{equation}\non
	\frac{1}{[s^{k}\gamma(s)]^{l-1}}\leq \frac{d(l-1)(s-v_*)}{s^{k(l-1)}}+\frac{1}{s^{k(l-1)}\gamma^{l-1}(v_*)}\rightarrow0,\;\;\;\text{as}\;s\rightarrow+\infty.
	\end{equation}This completes the proof.
\end{proof}
\begin{corollary}\label{cor1a}
	Assume that $n\geq3$, $\e\geq0$ and $\gamma(\cdot)$ satisfies $\mathrm{(A0)}$, $\mathrm{(A1)}$ and  $\mathrm{(A3u)}$. Then $v$ has a uniform-in-time upper bound in $\Omega\times[0,T_{\mathrm{max}}).$
\end{corollary}
\begin{proof}
	Note that $\frac{1}{l_0-1}<\frac{2}{n-2}$ when $l_0>\frac{n}{2}$. Thus $\gamma$ satisfies $\mathrm{(A2)}$ with some $k<\frac{2}{n-2}$ and due to Lemma \ref{lmvbd0} and Lemma \ref{lmvbd1}, $v$ has a uniform-in-time upper bound.
\end{proof}

\section{The parabolic-elliptic case}
This section is devoted to the proof of Theorem \ref{TH1}. With the upper bound of $v$ at hand, in view of Lemma \ref{criterion}, it suffices to establish an estimate for the weighted energy $\int_\Omega u^p\gamma^q(v)$ for some $p>\frac{n}{2}$ and $q>0.$
\subsection{Global existence}
 First, we prove existence of global classical solutions which is given by the following lemma.
\begin{lemma}
	Assume that $\e=0$ and $\gamma(\cdot)$ satisfies $\mathrm{(A0)}$ and 
	 $\mathrm{(A3a)}$. Then for any given $0<T<T_{\mathrm{max}}$, 
	 there {exist $p>\frac{n}{2}-1$ and $C_T>0$} such that
	\begin{equation*}
		\sup\limits_{0\leq t\leq T}\int_\Omega u^{1+p}\leq C_T,
	\end{equation*}where $p$ may depend on $n$, $\Omega$, $\gamma$ and $T$.
\end{lemma}
\begin{proof}
Recall that $v=w$ when $\e=0.$ Multiplying  the key identity \eqref{keyid} by $qu^{p+1}\gamma^{q-1}(v)\gamma'(v)$ with $p,q>0$ to be specified later and integrating with respect to $x$ yields
\begin{equation}\label{e3}
\begin{split}
\frac{d}{dt}\int_{\Omega}u^{p+1}\gamma^q(v)dx-(p+1)\int_{\Omega}\gamma^q(v)u^pu_tdx-q\int_{\Omega}u^{p+1}\gamma^q(v)\gamma'(v)\Delta vdx\\
-q\int_{\Omega}(I-\Delta)^{-1}[u\gamma(v)]  u^{p+1}\gamma^{q-1}(v)\gamma'(v)dx= -q\int_{\Omega}u^{p+1}\gamma^q(v)\gamma'(v)vdx,
\end{split}
\end{equation}
where we used the fact that $-\Delta v+v=u$.

By the first equation of \eqref{chemo1} and integration by parts, we infer that
\begin{equation}
\begin{split}\label{inte0}
&-(p+1)\int_{\Omega}\gamma^q(v)u^pu_tdx\\
=&-(p+1)\int_{\Omega}\gamma^q(v)u^p \Delta (\gamma (v)u)dx\\
=&(p+1)\int_{\Omega}\left(\gamma (v)\nabla u+\gamma'(v)u\nabla v\right)\left(pu^{p-1}\gamma^q(v)\nabla u+qu^p\gamma^{q-1}(v)\gamma'(v)\nabla v\right)dx\\
=&p(p+1)\int_{\Omega}u^{p-1}\gamma^{q+1}(v)|\nabla u|^2dx  +q(p+1)\int_{\Omega}u^{p+1}\gamma ^{q-1}(v)|\gamma'(v)|^2|\nabla v|^2dx\\
&+(p+1)(p+q)\int_{\Omega}u^p\gamma^q(v)\gamma'(v)\nabla u\cdot\nabla vdx, 
\end{split}
\end{equation}
and by integration by parts again,
\begin{equation}\label{e4}
\begin{split}
&-q\int_{\Omega}u^{p+1}\gamma^q(v)\gamma'(v)\Delta vdx\\
=&q^2\int_{\Omega} u^{p+1}\gamma^{q-1} (v)|\gamma'(v)|^2|\nabla v|^2dx+q\int_{\Omega}u^{p+1}\gamma^q\gamma''(v)|\nabla v|^2dx\\
&+q(p+1)\int_{\Omega}u^p\gamma^q (v)\gamma'(v) \nabla u\cdot\nabla vdx.
\end{split}
\end{equation}
Then we arrive at
\begin{equation}
\begin{split}\label{est0}
&\frac{d}{dt}\int_{\Omega}u^{p+1}\gamma^q(v)dx+(p+1)p\int_{\Omega}u^{p-1}\gamma^{q+1}|\nabla u|^2dx\\
&+q\int_{\Omega}\bigg((p+q+1)|\gamma'(v)|^2+\gamma\gamma''\bigg)u^{p+1}\gamma^{q-1}|\nabla v|^2dx\\
&-q\int_{\Omega}(I-\Delta)^{-1}[u\gamma(v)]  u^{p+1}\gamma^{q-1}(v)\gamma'(v)dx\\
=&-(p+1)(p+2q)\int_{\Omega}u^p\gamma^q(v)\gamma'(v)\nabla u\cdot\nabla vdx  -q\int_{\Omega}u^{p+1}\gamma^q(v)\gamma'(v)vdx.
\end{split}
\end{equation}
Now applying Young's inequality, we infer that
\begin{equation}
	\begin{split}\non
	&-(p+1)(p+2q)\int_{\Omega}u^p\gamma^q(v)\gamma'(v)\nabla u\cdot\nabla vdx\\
	&\leq (p+1)p\int_{\Omega}u^{p-1}\gamma^{q+1}|\nabla u|^2dx+\frac{(p+1)(p+2q)^2}{4p}\int_\Omega u^{1+p}\gamma^{q-1}|\gamma'|^2|\nabla v|^2dx.
	\end{split}
\end{equation}
We further require that
\begin{equation}
\begin{split}\label{cond00}
	&\frac{(p+1)(p+2q)^2}{4p}\int_\Omega u^{1+p}\gamma^{q-1}|\gamma'|^2|\nabla v|^2dx\\
	&\leq q\int_{\Omega}\bigg((p+q+1)|\gamma'(v)|^2+\gamma\gamma''\bigg)u^{p+1}\gamma^{q-1}|\nabla v|^2dx,
\end{split}
\end{equation}
which is satisfied provided that
\begin{equation}
	\begin{split}\label{cond0}
	(p^2+p^3+4q^2)|\gamma'|^2\leq 4pq\gamma\gamma''\qquad\text{a.e.}
	\end{split}
\end{equation}
Next, letting $q=\lambda p$ with some $\lambda>0$, then \eqref{cond0} is equivalent to the following
\begin{equation}\label{cond1}
	(1+p+4\lambda^2)|\gamma'|^2\leq 4\lambda\gamma\gamma''\qquad\text{a.e.}
\end{equation}
Note that $	\frac{1+p+4\lambda^2}{4\lambda}$ attains its minimum value $\sqrt{1+p}$ when $\lambda=\frac{\sqrt{1+p}}{2}$.

For any given $0<T<T_{\mathrm{max}}$,  due to Lemma \ref{lbdw} and Lemma \ref{keylem1}, $v$ is bounded on $[0,T]\times\overline{\Omega}$ from above and below by some  strictly positive constants depending only on the initial data, $\gamma$, $T$ and $\Omega$, which is also true for $|\gamma'(v)|^2$ and $\gamma(v)\gamma''(v)$ on $[0,T]\times\overline{\Omega}$ due to our assumption on $\gamma$. Then under the assumption $(\mathrm{A3a})$, one can always find $p>\frac{n}{2}-1$ and $\lambda=\frac{\sqrt{1+p}}{2}$ such that \eqref{cond1} holds on $[0,T]\times\overline{\Omega}$. As a result, 
one obtains that 
\begin{equation}
\begin{split}\label{est2}
&\frac{d}{dt}\int_{\Omega}u^{p+1}\gamma^q(v)dx-q\int_{\Omega}(I-\Delta)^{-1}[u\gamma(v)]  u^{p+1}\gamma^{q-1}(v)\gamma'(v)dx\\
\leq& -q\int_{\Omega}u^{p+1}\gamma^q(v)\gamma'(v)vdx.
\end{split}
\end{equation}
Then by Gronwall's inequality, we get 
\begin{equation}\non
	\int_\Omega u^{p+1}\gamma^q(v)dx\leq C_T.
\end{equation}
This concludes the proof since $\gamma^q(v)$ is bounded from below.
\end{proof}
\begin{corollary}\label{cor1} Assume  $\e=0$, $\gamma(v)=v^{-k}$ and $n\geq3$. Then there exists a unique global classical solution provided that $k<\frac{\sqrt{2n}+2}{n-2}$.
\end{corollary}

\subsection{Uniform-in-time boundedness}
In this part we prove the uniform-in-time boundedness in Theorem \ref{TH1}. To this aim, we  establish time-independent bounds for the weighted energy.
\begin{lemma}
Assume that $\e=0$, $\gamma(\cdot)$ satisfies $\mathrm{(A0)}$, $\mathrm{(A1)}$ and  $\mathrm{(A3u)}$. The there holds
\begin{equation}
	\sup\limits_{t\geq0}\int_\Omega u^pdx\leq C
\end{equation}with $p>\frac{n}{2}$ and $C>0$ depending only on the initial data, $\Omega$ and $\gamma.$
\end{lemma}
\begin{proof}
	Under our assumption, condition \eqref{cond1} holds for any $p>1$ such that
\begin{equation}\non
	\frac{1+p+4\lambda^2}{4\lambda}\leq l_0
\end{equation} holds with some $\lambda>0.$
Define \[f(\lambda)=4\lambda l_0-4\lambda^2\] for all $\lambda>0$. We observe that $f(\lambda)$ attains its maximum value $l_0^2$ at $\lambda_0=l_0/2$. Since $l_0>\frac{n}{2}$, there holds
	\begin{equation*}
	l_0^2>\frac{n^2}{4}.
	\end{equation*}
Thus, for any $1+p\in(1,\frac{n^2}{4}]$ and $\lambda=\lambda_0$, there holds
\begin{equation}
	1+p\leq\frac{n^2}{4}<l_0^2=f(\lambda_0).
\end{equation}
In other words, there holds
\begin{equation*}
\frac{1+p+4\lambda_0^2}{4\lambda_0}|\gamma'|^2< l_0|\gamma'|^2\leq \gamma\gamma'',\;\;\forall\;s>0
\end{equation*}
for any	$1+p\in(1,\frac{n^2}{4}]$. In particular, 
 recalling the time-independent lower and upper bounds for $v$  given by Corollary \ref{cor1a},
\begin{equation*}
v_*\leq v(x,t)\leq v^* \qquad\text{on}\;\;\overline{\Omega}\times [0,\infty)
\end{equation*}with $v_*,v^*>0$,
 we infer that
\begin{equation*}
\frac{1+p+4\lambda_0^2}{4\lambda_0}|\gamma'(v(x,t))|^2<  \gamma(v(x,t))\gamma''(v(x,t)),\;\;\text{on}\;\overline{\Omega}\times[0,\infty)
\end{equation*}for any	$1+p\in(1,\frac{n^2}{4}]$. In addition, for any	$1+p\in(1,\frac{n^2}{4}]$, we can further find time-independent $\delta_0=\delta_0(p,\lambda_0)>0$ such that
\begin{equation*}
\frac{1+p+4\lambda_0^2+4\lambda_0\delta_0(1+p+\lambda_0p)}{4\lambda_0(1-\delta_0)}|\gamma'(v(x,t))|^2<  \gamma(v(x,t))\gamma''(v(x,t)),\;\;\text{on}\;\overline{\Omega}\times[0,\infty).
\end{equation*}
As a result, based on a similar argument as from \eqref{cond00} to \eqref{cond1}, we have
\begin{equation}\non
	\frac{(p+1)(p+2q)^2}{4p(1-\delta_0)}\int_\Omega u^{1+p}\gamma^{q-1}|\gamma'|^2|\nabla v|^2\leq q\int_{\Omega}\bigg((p+q+1)|\gamma'(v)|^2+\gamma\gamma''\bigg)u^{p+1}\gamma^{q-1}|\nabla v|^2
\end{equation}with $q=\lambda_0p.$
Thus by Young's inequality,
\begin{equation}
\begin{split}
&-(p+1)(p+2q)\int_{\Omega}u^p\gamma^q(v)\gamma'(v)\nabla u\cdot\nabla vdx\\
&\leq (p+1)p(1-\delta_0)\int_{\Omega}u^{p-1}\gamma^{q+1}|\nabla u|^2dx\\&+\frac{(p+1)(p+2q)^2}{4p(1-\delta_0)}\int_\Omega u^{1+p}\gamma^{q-1}|\gamma'|^2|\nabla v|^2dx,
\end{split}
\end{equation} 
 we obtains an improved version of \eqref{est2} as follows
\begin{equation}
\begin{split}\label{est2b}
&\frac{d}{dt}\int_{\Omega}u^{p+1}\gamma^q(v)dx+\delta_0(p+1)p\int_{\Omega}u^{p-1}\gamma^{q+1}|\nabla u|^2dx\\
&\;\;-q\int_{\Omega}(I-\Delta)^{-1}[u\gamma(v)]  u^{p+1}\gamma^{q-1}(v)\gamma'(v)dx\\
\leq& -q\int_{\Omega}u^{p+1}\gamma^q(v)\gamma'(v)vdx
\end{split}
\end{equation}with any $1+p\in(1,\frac{n^2}{4}]$, $q=\frac{pl_0}{2}$ and some $\delta_0=\delta_0(p,l_0)>0.$

Now,  recalling Lemma \ref{wH1b} and the time-independent boundedness of $v$, there holds
\begin{equation}\label{est3}
\sup\limits_{t\geq0}\int_t^{t+1}\int_\Omega u^2dxds\leq C
\end{equation}with $C>0$ depending only on the initial data, $\Omega$ and $\gamma$.

Next, we take $p=1$ such that $1+p=2<\frac{n^2}{4}$ and $q=\frac{l_0}{2}$ in \eqref{est2b}. Since now $v$  is bounded from above and below, we obtain that
\begin{equation}\label{est2bb}
	\frac{d}{dt}\int_\Omega u^{2} v^{-\frac{l_0}{2}}dx+C\int_\Omega |\nabla u|^2dx\leq C\int_\Omega u^{2}dx
\end{equation}with $C>0$ independent of time.

In view of \eqref{est3}, an application of the uniform Gronwall inequality together with the local boundedness yields that
\begin{equation}\non
\sup\limits_{t\geq0}\int_\Omega u^2dx\leq C.
\end{equation}
Besides, an integration of \eqref{est2bb}  from $t$ to $t+1$ further gives rise to
\begin{equation}\non
\sup\limits_{t\geq0}	\int_t^{t+1}\int_\Omega|\nabla u|^2dxds\leq C.
\end{equation}
Thus, by the Sobolev embedding 
\begin{equation}\label{emb0}
	\|\xi\|_{L^{r_*}(\Omega)}^{r_*}\leq C\|\nabla \xi\|_{L^2(\Omega)}^2\|\xi\|_{L^2(\Omega)}^{r_*-2}+C\|\xi\|^{r_*}_{L^1(\Omega)}
\end{equation}
with $r_*=2+\frac{4}{n}$, we infer that
\begin{equation}\non
\sup\limits_{t\geq0}\int_t^{t+1}\int_\Omega u^{r_*}\leq C.
\end{equation}

Then, we divide the discussion into several cases regarding the spatial dimensions. First, when $n=3$, one notes that $r_*=2+\frac{4}{n}=\frac{10}{3}>\frac{9}{4}=\frac{n^2}{4}$ and  we may take $1+p=\frac{9}{4}$ and $q=\frac{5l_0}{8}$ in \eqref{est2b}. In the same manner as before, by the uniform  Gronwall inequality, we deduce that
\begin{equation}
\sup\limits_{t\geq0}\int_\Omega u^{\frac{9}{4}}dx\leq C.\non
\end{equation}

When $n\geq4$, there holds $r_*=2+\frac{4}{n}<\frac{n^2}{4}$. We can take $1+p=r_*$ and $q=\frac{l_0}{2}(r_*-1)$ in \eqref{est2b} to obtain in the same manner as before that
\begin{equation}\label{estb0}
	\sup\limits_{t\geq0}\left(\int_\Omega u^{r_*}+\int_t^{t+1}\int_\Omega |\nabla u^{\frac{r_*}{2}}|^2dxds\right)\leq C.
\end{equation}
Note that when $n=4,5$, we have $r_*=2+\frac{4}{n}>\frac{n}{2}$.

It remains to consider the case $n\geq6$. First,  using the embedding \eqref{emb0} with $\xi=u^{\frac{r_*}{2}}$,  we infer from \eqref{estb0} that \begin{equation}\non
\sup\limits_{t\geq0}\int_t^{t+1}\int_\Omega u^{\frac{r_*^2}{2}}dxds\leq C.
\end{equation} 
On the other hand when $n\geq6$, one can always find   $m\in\mathbb{N}$ such that $\frac{n}{2}<\frac{r_*^m}{2}\leq \frac{n^2}{4}$. Indeed, let $m$ be the integer such that $\frac{r_*^m}{2}\leq\frac{n^2}{4}$ and $\frac{r_*^{m+1}}{2}>\frac{n^2}{4}$. Then we observe that
\begin{equation}\non
\frac{r_*^m}{2}>\frac{n^2}{4r_*}=\frac{n}{2}\times\frac{n}{2r_*}
\end{equation}
where $\frac{n}{2r_*}>1$ if $n\geq6.$

Using the embedding \eqref{emb0} with $\xi=u^{\frac{r_*^{l}}{2}}$ with $l=1,2,...,m-1$, repeating the above steps, we can finally prove that
\begin{equation}\non
	\sup\limits_{t\geq0}\int_\Omega u^{\frac{r_*^m}{2}}dx\leq C.
\end{equation}
This completes the proof.
\end{proof}

\section{The fully parabolic case}
In this section, we consider the fully parabolic case and give a proof for Theorem \ref{TH2}. The idea is basically a generalization of  \cite[Lemma 5.5]{FJ19b} to higher dimensions. Indeed, we list out a system of estimations involving the weighted energies $\int u^{1+p}\gamma^q(v)$ with the varying parameters $p,q$. Luckily, by a careful recombination we are able to obtain the uniform-in-time boundedness.

\begin{lemma}\label{lemn3exist}
	Assume $n\geq3$.  Suppose that  $\gamma(\cdot)$ satisfies  $\mathrm{(A0)}$,  $\mathrm{(A1)}$,  and  $\mathrm{(A3b)}$. Then there is $C>0$ depending only on the initial data and $\Omega$ such that
	\begin{equation*}
	\sup\limits_{0\leq t<T_{\mathrm{max}}}\int_\Omega u^{1+[\frac{n}{2}]}dx\leq C.
	\end{equation*}
\end{lemma}

\begin{proof}In the same manner as before, we first compute by integration by parts to obtain that
	\begin{equation}
		\begin{split}\non
		&\frac{d}{dt}\int_\Omega u^{1+p}\gamma^q(v)dx\\
		=&(1+p)\int_\Omega u^p\gamma^q(v)u_t+q\int_\Omega u^{1+p}\gamma^{q-1}(v)\gamma'(v)v_t\\
		=&(1+p)\int_\Omega u^p\gamma^q(v)\Delta (u\gamma(v))+q\int_\Omega u^{1+p}\gamma^{q-1}(v)\gamma'(v)(u-v+\Delta v)\\
		=&-(1+p)\int_\Omega \nabla (u^p\gamma^q(v))\cdot\nabla (u\gamma(v))+q\int_\Omega u^{2+p}\gamma^{q-1}(v)\gamma'(v)-q\int_\Omega u^{1+p}\gamma^{q-1}(v)\gamma'(v)v\\
		&\qquad-q\int_\Omega \nabla(u^{1+p}\gamma^{q-1}(v)\gamma'(v))\cdot\nabla v.
		\end{split}
	\end{equation}The main difference here is that we need to use the second equation in \eqref{chemo1} to replace $v_t$.
	
Recalling \eqref{inte0}, 
\begin{equation}
	\begin{split}\non
	&(1+p)\int_\Omega \nabla (u^p\gamma^q(v))\cdot\nabla (u\gamma(v))\\
	=&p(1+p)\int_\Omega u^{p-1}\gamma^{1+q}|\nabla u|^2+q(1+p)\int_\Omega u^{1+p}\gamma^{q-1}|\gamma'|^2|\nabla v|^2\\
	&\qquad+(1+p)(p+q)\int_\Omega u^p\gamma^q\gamma'\nabla u\cdot\nabla v,
	\end{split}
\end{equation}
and by integration by parts again,
\begin{equation}
	\begin{split}\non
	&q\int_\Omega \nabla(u^{1+p}\gamma^{q-1}(v)\gamma'(v))\cdot\nabla v\\
	=&q(1+p)\int_\Omega u^p\gamma^{q-1}\gamma'\nabla u\cdot\nabla v+q\int_\Omega u^{1+p}\gamma^{q-2}\bigg((q-1)|\gamma'|^2+\gamma\gamma''\bigg)|\nabla v|^2.
	\end{split}
\end{equation}
As a result, we obtain that
\begin{equation}
	\begin{split}\label{kestb0}
	&\frac{d}{dt}\int_\Omega u^{1+p}\gamma^q(v)dx+p(1+p)\int_\Omega u^{p-1}\gamma^{1+q}|\nabla u|^2+q(1+p)\int_\Omega u^{1+p}\gamma^{q-1}|\gamma'|^2|\nabla v|^2\\
	&+q \int_\Omega u^{1+p}\gamma^{q-2}\bigg((q-1)|\gamma'|^2+\gamma\gamma''\bigg)|\nabla v|^2-q\int_\Omega u^{2+p}\gamma^{q-1}(v)\gamma'(v)\\
	=&-(1+p)(p+q)\int_\Omega u^p\gamma^q\gamma'\nabla u\cdot\nabla v-q(1+p)\int_\Omega u^p\gamma^{q-1}\gamma'\nabla u\cdot\nabla v\\
	&\qquad-q\int_\Omega u^{1+p}\gamma^{q-1}(v)\gamma'(v)v.
	\end{split}
\end{equation}
In particular, if $p=q$, since
\begin{equation}
\begin{split}\non
&(1+p)\int_\Omega \nabla (u^p\gamma^p(v))\cdot\nabla (u\gamma(v))\\
=&p(1+p)\int_\Omega (u\gamma)^{p-1}|\nabla (u\gamma)|^2,
\end{split}
\end{equation}
one obtains the following estimate
\begin{equation}
\begin{split}\nonumber
&\frac{d}{dt}\int_\Omega u^{1+p}\gamma^p(v)dx+p(1+p)\int_\Omega (u\gamma)^{p-1}|\nabla (u\gamma)|^2\\
&\qquad +p \int_\Omega u^{1+p}\gamma^{p-2}\bigg((p-1)|\gamma'|^2+\gamma\gamma''\bigg)|\nabla v|^2-p\int_\Omega u^{2+p}\gamma^{p-1}(v)\gamma'(v)\\
=&-p(1+p)\int_\Omega u^p\gamma^{p-1}\gamma'\nabla u\cdot\nabla v-p\int_\Omega u^{1+p}\gamma^{p-1}(v)\gamma'(v)v.
\end{split}
\end{equation}

Now, let $1\leq p\leq[\frac{n}{2}]$ be any integer fixed. In addition, assume  $q=j$ in \eqref{kestb0} with $j=0,1,2,...,p$ and  multiply the $j$th formula by $\lambda_{p,j}>0$ to be specify later. Then a summation yields that
\begin{equation}
\begin{split}
&\frac{d}{dt}\sum_{j=0}^{j=p}\lambda_{p,j}\int_\Omega u^{1+p}\gamma^j(v)dx+\lambda_{p,p}p(1+p)\int_\Omega (u\gamma)^{p-1}|\nabla(u\gamma)|^{2}\\
&+\left\{\int_\Omega u^{1+p}\gamma^{p-2}\bigg((p-1)\big(p\lambda_{p,p}+(p+1)\lambda_{p-1}\big)|\gamma'|^2+p\lambda_{p,p}\gamma\gamma''\bigg)|\nabla v|^2\right.\\
&\;\;\left.+\lambda_{p,p-1}p(p+1)\int_\Omega u^{p-1}\gamma^p|\nabla u|^2+(1+p)\big(p\lambda_{p,p}+(2p-1)\lambda_{p,p-1}\big)\int_\Omega u^p\gamma^{p-1}\gamma'\nabla u\cdot\nabla v\right\}\\
&+\left\{\int_\Omega u^{1+p}\gamma^{p-3}\bigg((p-2)\big((p-1)\lambda_{p,p-1}+(p+1)\lambda_{p,p-2}\big)|\gamma'|^2+(p-1)\lambda_{p,p-1}\gamma\gamma''\bigg)|\nabla v|^2\right.\\
&\left.+\lambda_{p,p-2}p(p+1)\int_\Omega u^{p-1}\gamma^{p-1}|\nabla u|^2+(1+p)\big((p-1)\lambda_{p,p-1}+(2p-2)\lambda_{p,p-2}\big)\int_\Omega u^p\gamma^{p-2}\gamma'\nabla u\cdot\nabla v\right\}\\
&+...\\
&+\left\{\lambda_{p,1}\int_\Omega u^{1+p}\gamma''|\nabla v|^2+\lambda_{p,0}p(p+1)\int_\Omega u^{p-1}\gamma|\nabla u|^2\right.\\
&\left.+(1+p)\big(\lambda_{p,1}+p\lambda_{p,0}\big)\int_\Omega u^p\gamma'\nabla u\cdot\nabla v\right\}\\
&-\sum_{j=1}^{j=p}j\lambda_{p,j}\int_\Omega u^{2+p}\gamma'(v)\gamma^{j-1}\\
=&-\sum_{j=1}^{j=p}j\lambda_{p,j}\int_\Omega u^{1+p}\gamma'(v)\gamma^{j-1}v.
\end{split}\non
\end{equation}
For $1\leq j\leq p$, we define
\begin{equation}\non
\begin{split}
\Lambda_{p,j}=&\int_\Omega u^{1+p}\gamma^{j-2}\bigg((j-1)\big(j\lambda_{p,j}+(p+1)\lambda_{p,j-1}\big)|\gamma'|^2+j\lambda_{p,j}\gamma\gamma''\bigg)|\nabla v|^2\\
&+\lambda_{p,j-1}p(p+1)\int_\Omega u^{p-1}\gamma^{j}|\nabla u|^2+(1+p)\big(j\lambda_{p,j}+(p+j-1)\lambda_{p,j-1}\big)\int_\Omega u^p\gamma^{j-1}\gamma'\nabla u\cdot\nabla v.
\end{split}
\end{equation}
Then we obtain that
\begin{equation}
	\begin{split}\label{kest222}
	&\frac{d}{dt}\sum_{j=0}^{j=p}\lambda_{p,j}\int_\Omega u^{1+p}\gamma^j(v)dx+\lambda_{p,p}p(1+p)\int_\Omega (u\gamma)^{p-1}|\nabla(u\gamma)|^{2}\\
	&+\sum_{j=1}^{j=p}\Lambda_{p,j}-\sum_{j=1}^{j=p}j\lambda_{p,j}\int_\Omega u^{2+p}\gamma'(v)\gamma^{j-1}\\
	=&-\sum_{j=1}^{j=p}j\lambda_{p,j}\int_\Omega u^{1+p}\gamma'(v)\gamma^{j-1}v.
	\end{split}
\end{equation}
Invoking Young's inequality, we infer that
\begin{equation}
	\begin{split}\non
	&(1+p)\big(j\lambda_{p,j}+(p+j-1)\lambda_{p,j-1}\big)\int_\Omega u^p\gamma^{j-1}\gamma'\nabla u\cdot\nabla v\\
	\leq&\lambda_{p,j-1}p(p+1)\int_\Omega u^{p-1}\gamma^{j}|\nabla u|^2+\frac{(1+p)\bigg[j\lambda_{p,j}+(p+j-1)\lambda_{p,j-1}\bigg]^2}{4p\lambda_{p,j-1}}\int_\Omega u^{1+p}\gamma^{j-2}|\gamma'|^2|\nabla v|^2.
	\end{split}
\end{equation}
Hence, in the same spirit as before, we have $\Lambda_{p,j}\geq0$ provided that for all $s>0$
\begin{equation}
	\begin{split}\non
&\frac{(1+p)\bigg[j\lambda_{p,j}+(p+j-1)\lambda_{p,j-1}\bigg]^2}{4p\lambda_{p,j-1}}|\gamma'(s)|^2\\
\leq& (j-1)\big(j\lambda_{p,j}+(p+1)\lambda_{p,j-1}\big)|\gamma'(s)|^2+j\lambda_{p,j}\gamma(s)\gamma''(s),
	\end{split}
\end{equation}
which by simple computations is equivalent to
\begin{equation}
	\begin{split}\label{condition2}
	\frac{(1+p)j^2\lambda_{p,j}^2+(1+p)(p+1-j)^2\lambda_{p,j-1}^2+2j\lambda_{p,j-1}\lambda_{p,j}(p^2-pj+2p+j-1)}{4pj\lambda_{p,j-1}\lambda_{p,j}}|\gamma'|^2\leq\gamma\gamma''.
	\end{split}
\end{equation}
Observe that by Young's inequality again,
\begin{equation}
	\begin{split}\non
	&(1+p)j^2\lambda_{p,j}^2+(1+p)(p+1-j)^2\lambda_{p,j-1}^2+2j\lambda_{p,j-1}\lambda_{p,j}(p^2-pj+2p+j-1)\\
	&\geq4pj\lambda_{p,j-1}\lambda_{p,j}(p+2-j)
	\end{split}
\end{equation}
where the minimum is attained provided that
\begin{equation}\non
	j\lambda_{p,j}=(p+1-j)\lambda_{p,j-1}.
\end{equation}
Thus, if we take $\lambda_{p,0}=1$ and $\lambda_{p,j}=(p+1-j)\lambda_{j-1}/j$ for $1\leq j\leq p$, condition \eqref{condition2} reads
\begin{equation}
	(p+2-j)|\gamma'|^2\leq \gamma\gamma'',\;\;\forall \;s>0\;\text{and for all}\;1\leq j\leq p.
\end{equation}
Thus, under the assumption $\mathrm{(A3b)}$, one can always find $\lambda_{p,j}>0$ such that $\Lambda_{pj}\geq0$ for any fixed $1\leq p\leq [\frac{n}{2}]$ with all $1\leq j\leq p$. As a result, there holds
\begin{equation}
\begin{split}\label{kest222c}
&\frac{d}{dt}\sum_{j=0}^{j=p}\lambda_{p,j}\int_\Omega u^{1+p}\gamma^j(v)dx+\lambda_{p,p}p(1+p)\int_\Omega (u\gamma)^{p-1}|\nabla(u\gamma)|^{2}-\sum_{j=1}^{j=p}j\lambda_{p,j}\int_\Omega u^{2+p}\gamma'(v)\gamma^{j-1}\\
\leq&-\sum_{j=1}^{j=p}j\lambda_{p,j}\int_\Omega u^{1+p}\gamma'(v)\gamma^{j-1}v.
\end{split}
\end{equation}

Now, we recall that $v$ is uniformly-in-time bounded from above and below under the assumption $\mathrm{(A3b)}$                                       . Moreover by Lemma \ref{wH1b},
\begin{equation}\non
\int_t^{t+\tau}\int_\Omega u^2dxds\leq C
\end{equation}with $C>0$ independent of time. We can first take $p=1$ in \eqref{kest222c} and use the uniform Gronwall inequality together with the above estimates to derive that
\begin{equation}
\sup\limits_{0\leq t<T_{\mathrm max}}\int_\Omega u^2dx\leq C.\non
\end{equation}
Moreover thanks to the third term on the left-hand side of \eqref{kest222c}, there holds
\begin{equation}\non
\int_t^{t+\tau}\int_\Omega u^3dxds\leq C.
\end{equation}
Subsequently, in the same manner as above, we can deduce by iterations that
\begin{equation}
\sup\limits_{0\leq t<T_{\mathrm max}}\int_\Omega u^{1+[\frac{n}{2}]}dx\leq C\non.
\end{equation}
\end{proof}
\begin{remark}
	Our assumption $\mathrm{(A3b)}$ is independent of the coefficients of the system. Indeed, if we replace the second equation of system \eqref{chemo1} by $v_t-\alpha\Delta v+\beta v=\theta u$ with some $\alpha,\beta,\theta>0$, one easily checks that condition \eqref{condition2} becomes
	\begin{equation}\non
		\begin{split}\label{condition2b}
	\frac{(1+p)\alpha^2j^2\lambda_j^2+(1+p)(p+1-j)^2\lambda_{j-1}^2+2\alpha j\lambda_{j-1}\lambda_j(p^2-pj+2p+j-1)}{4\alpha pj\lambda_{j-1}\lambda_j}|\gamma'|^2\leq\gamma\gamma'',
	\end{split}
	\end{equation}
which still yields to
\begin{equation}\non
(p+2-j)|\gamma'|^2\leq \gamma\gamma'',\;\;\forall \;s>0\;\text{and for all}\;1\leq j\leq p,
\end{equation}
 if we take $\lambda_0=1$ and $\lambda_j=\frac{(p+1-j)\lambda_{j-1}}{j\alpha}$ for $1\leq j\leq p$.
\end{remark}

\bigskip
\noindent\textbf{Acknowledgments} \\
K. Fujie is supported by Japan Society for the Promotion of Science (Grant-in-Aid for Early-Career Scientists; No.\ 19K14576).

\end{document}